\documentclass[11pt]{article}
\usepackage{authblk}
\usepackage{hyperref}
\usepackage[latin1]{inputenc}
\usepackage{bbm}
\usepackage{stmaryrd}
\usepackage{latexsym}
\usepackage{amsfonts}
\usepackage{amsmath,amssymb,amscd}
\usepackage{yfonts}
\usepackage{dsfont}
\usepackage[dvips]{graphicx}
\usepackage{epsfig}
\usepackage{psfrag}
\usepackage{subfigure}
\usepackage[font={small},margin=20pt]{caption}

\DeclareFontFamily{U}{txsyc}{}
\DeclareFontShape{U}{txsyc}{m}{n}{
   <-> txsyc%
}{}
\DeclareFontShape{U}{txsyc}{bx}{n}{
   <-> txbsyc%
}{}
\DeclareFontShape{U}{txsyc}{l}{n}{<->ssub * txsyc/m/n}{}
\DeclareFontShape{U}{txsyc}{b}{n}{<->ssub * txsyc/bx/n}{}
\DeclareSymbolFont{symbolsC}{U}{txsyc}{m}{n}
\SetSymbolFont{symbolsC}{bold}{U}{txsyc}{bx}{n}
\DeclareFontSubstitution{U}{txsyc}{m}{n}
\DeclareMathSymbol{\df}{\mathrel}{symbolsC}{"42}
\DeclareMathSymbol{\fd}{\mathrel}{symbolsC}{"43}
\DeclareMathSymbol{\lJoin}{\mathrel}{symbolsC}{"58}
\DeclareMathSymbol{\rJoin}{\mathrel}{symbolsC}{"59}

\newcommand{\cB}{{\cal B}}

\newcommand{\cpn}{{\check{p}_n}}

\newcommand{\dnj}{\delta_{n,j}}
\newcommand{\dnl}{\delta_{n,l}}

\newcommand{\EE}{\mathbb{E}}

\newcommand{\Hnu}{\mathbf{H}_{\nu}^{+}}
\newcommand{\Hs}{\mathbf{H}_{s}(R)}

\newcommand{\LL}{\mathbb{L}}
\newcommand{\NN}{\mathbb{N}}
\newcommand{\PP}{\mathbb{P}}

\newcommand{\hPZ}{\hat{\Phi}_{Z,n}}
\newcommand{\hpz}[1]{\hat{\Phi}_{Z,n}^{(#1)}}
\newcommand{\PZ}{\Phi_Z}

\newcommand{\PU}{\Phi_U}

\newcommand{\PX}{\Phi_X}
\newcommand{\hPXn}[1]{\hat{\Phi}_{X_{n,#1}}^{(2)}}
\newcommand{\hPXnj}{\hat{\Phi}_{X_{n,j}}^{(2)}}
\newcommand{\hPXnl}{\hat{\Phi}_{X_{n,l}}^{(2)}}

\newcommand{\PK}{\Phi_K}

\newcommand{\hpjs}{\hat{p}_{n,j^\star}}
\newcommand{\hpln}{\hat{p}^{(1)}_{n,l}}

\newcommand{\hjp}{\hat{j}_n^p}

\newcommand{\hpnj}{\hat{p}_{n,j}}
\newcommand{\hpjn}{\hat{p}_{n,\hjp}}

\newcommand{\hjf}{\hat{j}_n^f}
\newcommand{\hfjn}{\hat{f}_{n,\hjf}}

\newcommand{\hfnj}{\hat{f}_{n,j}}

\newcommand{\PQ}{\Phi_Q}

\newcommand{\RR}{\mathbb{R}}

\newcommand{\ii}{\mathfrak{i}}

\newcommand{\me}{\medskip}

\newcommand{\un}{\mathds{1}}

\newcommand{\bq}{\begin{eqnarray*}}
\newcommand{\bqn}[1]{\begin{eqnarray}\label{#1}}
\newcommand{\eq}{\end{eqnarray*}}
\newcommand{\eqn}{\end{eqnarray}}

\newcommand{\thistitlepagestyle}{}

\newcommand{\ttsim}{\raise.17ex\hbox{$\scriptstyle\mathtt{\sim}$}}

\newcommand{\VV}{\mathbb{V}}
\newtheorem{pro}{Proposition} 

\newtheorem{lem}[pro]{Lemma}
\newtheorem{theo}[pro]{Theorem}
\renewcommand{\thepro}{\arabic{pro}}

\newenvironment{rem}
{\par\me\refstepcounter{pro}\noindent{\bf Remark \thepro\ }}
{\par\hfill \par\me\noindent}

\newcommand{\ben}{\vspace{0mm}\begin{equation}}
\newcommand{\een}{\vspace{0mm}\end{equation}}
\newcommand{\be}{\vspace{0mm}\begin{equation*}}
\newcommand{\ee}{\vspace{0mm}\end{equation*}}
\newcommand{\bea}{\vspace{0mm}\begin{equation*}\begin{aligned}}
\newcommand{\eea}{\vspace{0mm}\end{aligned}\end{equation*}}
\newcommand{\bean}{\vspace{0mm}\begin{equation}\begin{aligned}}
\newcommand{\eean}{\vspace{0mm}\end{aligned}\end{equation}}

\newcommand{\E}{\mathbb{E}}
\newcommand{\Pro}{\mathbb{P}}

\newcommand{\R}{\mathbb{R}}

\newcommand{\abso}[1]{\bigl| #1\bigr|}

\usepackage{color}

\title{Cytometry inference through adaptive atomic deconvolution}
\author[1]{Manon Costa }
\author[2]{Sébastien Gadat}
\author[3]{Pauline Gonnord}
\author[1]{Laurent Risser}

\affil[1]{\small{
Institut de Mathématiques de Toulouse, UMR 5219\\
Université de Toulouse III, France.
} }
\affil[2]{Toulouse School of Economics, UMR 5604\\
Université de Toulouse 1,  France.}
\affil[3]{Centre de Physiopathologie Toulouse Purpan (CPTP), INSERM UMR1043, CNRS UMR 5282. Université Toulouse III, France.}

\date{}

\begin{document}

\setbox3=\vbox{
\vskip2mm
\hbox{${}^\dagger$ Institut de Mathématiques de Toulouse\\}
\hbox{Université Toulouse 3 Paul Sabatier\\}
\hbox{118, route de Narbonne\\} 
\hbox{31062 Toulouse Cedex 9, France\\}
\vskip2mm
\hbox{$^\ddagger$Toulouse School of Economics\\}
\hbox{Université Toulouse 1 Capitole \\}
\hbox{21, allées de Brienne\\} 
\hbox{31000 Toulouse, France\\}

\vskip2mm
\hbox{${}^{\ddagger\ddagger}$ Centre de Physiopathologie Toulouse Purpan (CPTP), \\}
\hbox{INSERM UMR1043, CNRS UMR 5282, \\}
\hbox{Université Toulouse III, France\\} 
\hbox{CHU Purpan, 
BP 3028, 31024 Toulouse CEDEX 03, France\\}
\vskip2mm
\hbox{\{costa${}^\dagger$,gadat$^\ddagger$,risser${}^\dagger$\}@math.univ-toulouse.fr\\}
\hbox{pauline.gonnord${}^{\ddagger\ddagger}$@inserm.fr\\}
}
\setbox5=\vbox{
\box3
}

\maketitle
\thistitlepagestyle
\abstract{ In this paper we consider a statistical estimation problem known as atomic deconvolution. Introduced in reliability, this model has a direct application when considering biological data produced by flow cytometers. In these experiments, biologists measure the fluorescence emission of treated cells and compare them with their natural emission to study the presence of specific molecules on the cells' surface. They observe a signal which is composed of a noise (the natural fluorescence) plus some additional signal related to the quantity of molecule present on the surface if any. From a statistical point of view, we aim at inferring the percentage of cells expressing the selected molecule and the probability distribution function associated with its fluorescence emission. We propose here an adaptive estimation procedure based on a previous deconvolution procedure introduced by  \cite{vanes1,vanes2}. For both estimating the mixing parameter and the mixing density automatically, we use the Lepskii method based on the optimal choice of a bandwidth using a bias-variance decomposition. We then derive some concentration inequalities for our estimators and obtain the convergence rates, that are shown to be minimax optimal (up to some $\log$ terms) in Sobolev classes. Finally, we apply our algorithm on simulated and real biological data.
}
\bigskip\\

{\small
}

\section{Introduction}

\subsection{Motivation}
This paper deals with a statistical estimation problem close to the standard deconvolution problem in density estimation, and known as the problem of \textit{atomic deconvolution}. This problem has been recently introduced in \cite{vanes1} and motivated by a reliability problem estimation. We consider in this work a natural application of this model to the biologocial datasets produced by flow cytometers. To motivate our theoretical and practical study, we have chosen to first introduce a common problem that should be addressed by biologist researchers when using flow cytometry measurements. A flow cytometer is an electronic machine that makes it possible to analyse a large number of cells and produce cell engineering results such as counting, sorting or bio-marker detection. 
This technology is commonly used for measuring the expression levels of proteins on the cells' surface. To this aim, the biologist use fluorescent antibodies which bind with specific proteins on the surface of cells. The flow cytometer then measures on each cells the fluorescent intensity, which is a reflect of the quantity of marker expressed by the cell. More precisely these measurements derive from a standard procedure in flow cytometry :
\begin{itemize}
\item First the biologist performs a \textit{calibration} that corresponds to the preliminary estimation of the baseline population of cells without any treatment: a large number of cells is placed in the cytometer and a fluorescent intensity is measured on each cell. This calibration ends with a baseline estimation of the the baseline photon emission by a population of cells.
\item Then the biologist exposes the cells to a fluorescent antibody, which binds with the marker of interest on the surface of cells. The new fluorescence empirical distribution is built by the cytometer.
\item The expertise of the biologist is at last used to calibrate a qualitative analysis to decide if there is an effect (or not) of the treatment and what is its mean effect.
We should mention that this human expertise may be also assisted by some recent statistical computational analyses and have resulted in an open project called \textit{FlowCAP} that makes it possible to produce standard statistical analysis (data clustering for example).
\end{itemize}
%
%
Figure  \ref{fig:cyto_exemple} is an illustration of the empirical distribution of flurescence measured by a cytometer before and after a treatment on different kind of cells with a logarithmic scale.
\begin{figure}
\hfill
\begin{minipage}{0.45\linewidth}
\centering
\includegraphics[scale=0.5]{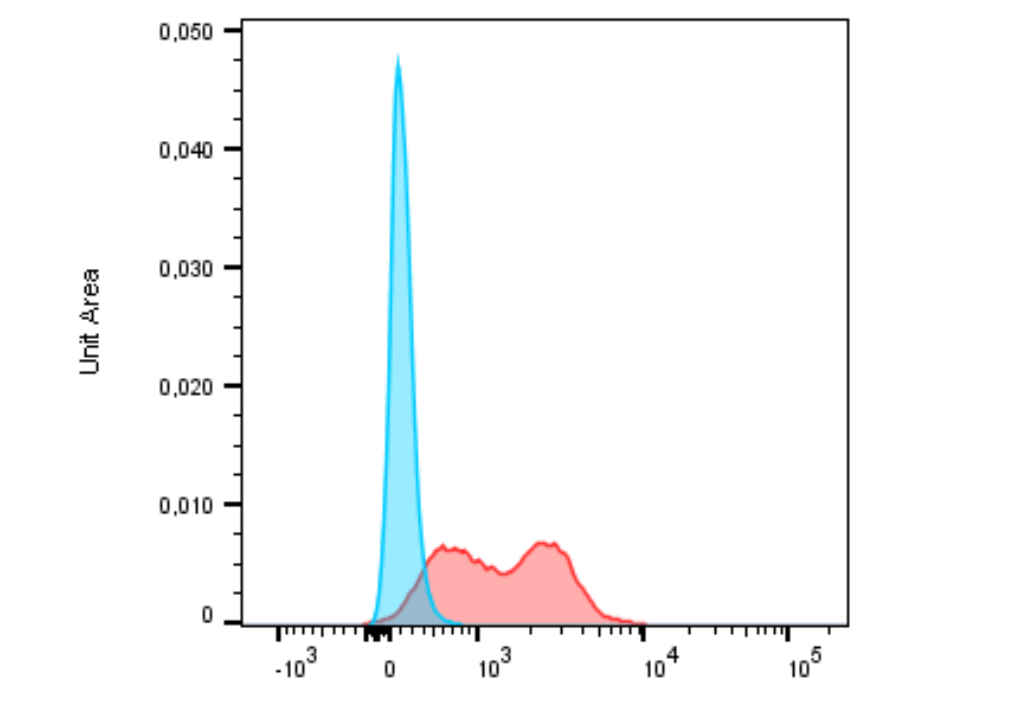}\\
\small{(a)}
\end{minipage}\hfill
\begin{minipage}{0.45\linewidth}
\centering
\includegraphics[scale=0.5]{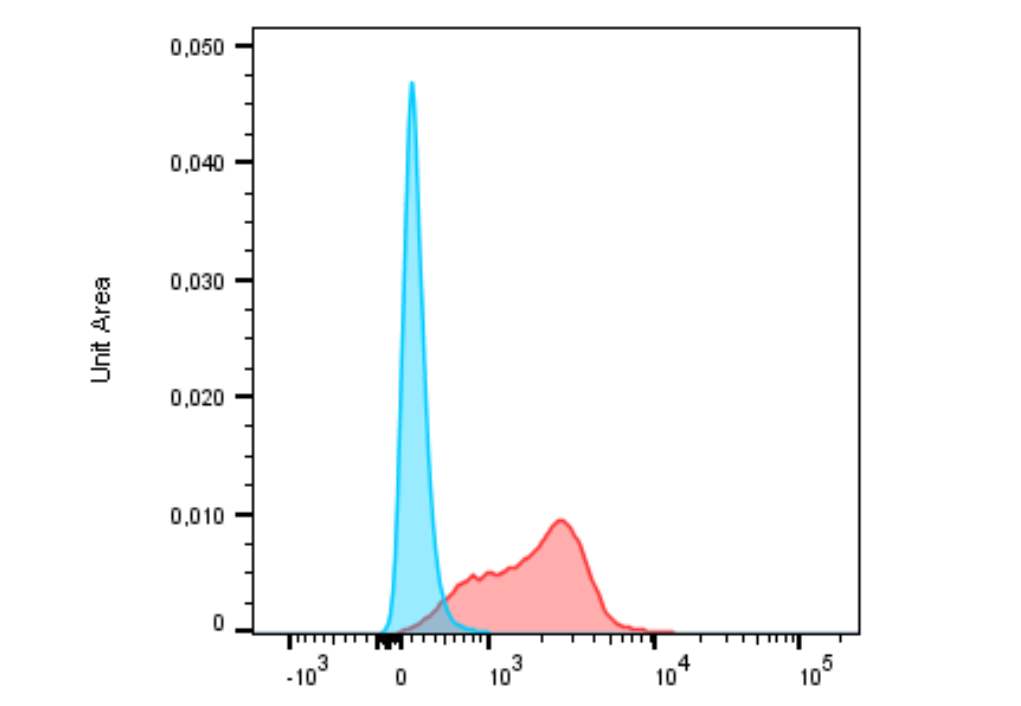}\\
\small{(b)}
\end{minipage}
\hfill

\hfill
\begin{minipage}{0.45\linewidth}
\centering
\includegraphics[scale=0.5]{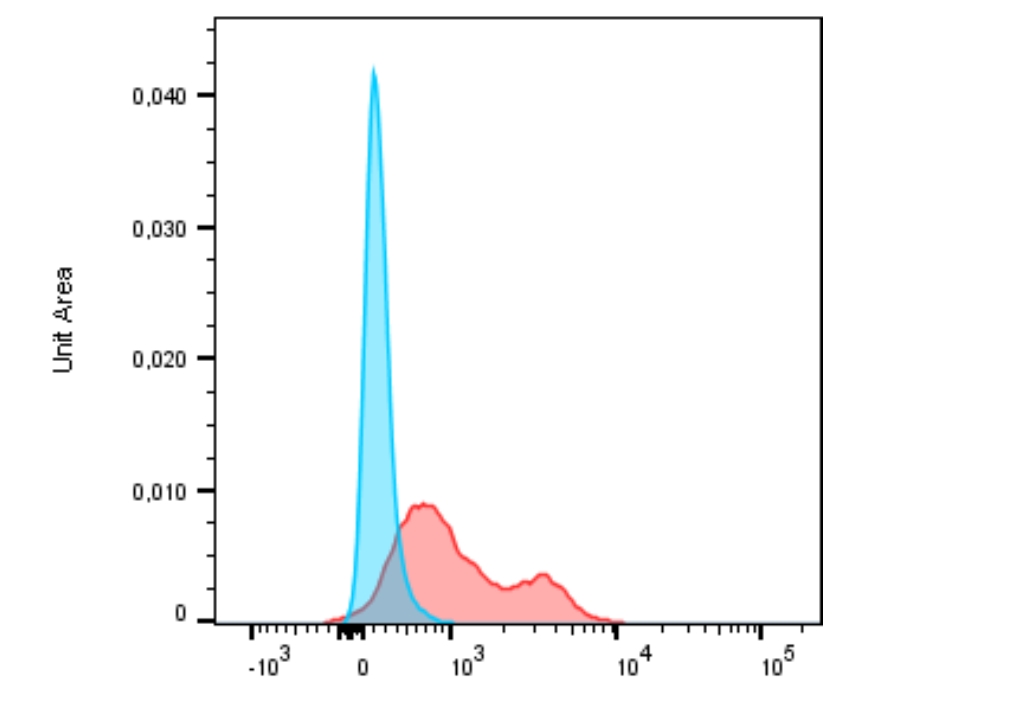}\\
\small{(c)}
\end{minipage}
\hfill
\begin{minipage}{0.45\linewidth}
\centering
\includegraphics[scale=0.5]{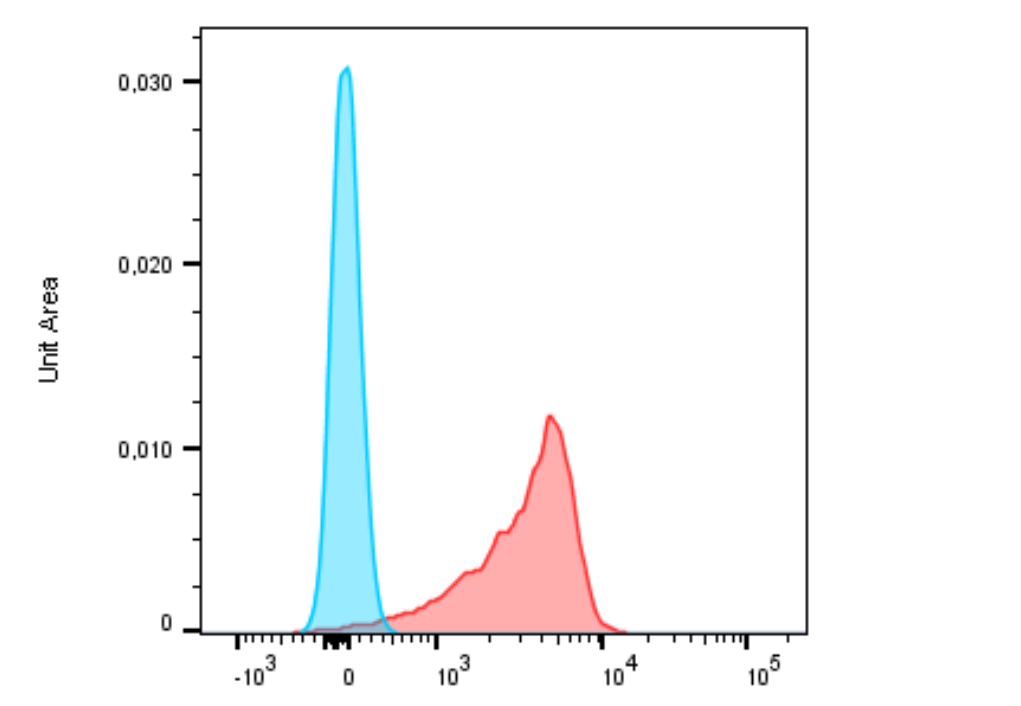}\\
\small{(d)}
\end{minipage}
\hfill
\caption{Effect of binding of an antibody specific for the CD27 protein expressed at the cell surface of lymphocytes extracted from the blood of an healthy individual (four different samples). The blue histograms represent the distribution of the baseline photo emission of untreated cells while the red histograms gives the fluorescent distribution of treated cells. We refer to Section \ref{subsec:realdata} for more details on the samples\label{fig:cyto_exemple}.}
\end{figure}

\vspace{1em}
Our objective here is to enrich the possibilities of automatic estimation on flow cytometer datasets through a statistical procedure that permit to retrieve  both the proportion of positive cells (\textit{i.e.} the cells that express a given marker) and the distribution of the marker in a population of cells. 

\subsection{Model}\label{sec:model}

The above phenomenon can be described as the following atomic deconvolution problem: we observe some  i.i.d. realizations $(Z_i)_{1 \leq i \leq n}$ of a model given by:
\begin{equation}\label{eq:model}
Z=U+A X,
\end{equation}
where $U,A$ and $X$ are jointly independent random variables. In Equation \eqref{eq:model}, the random variable $Z$ is therefore the measurement of the flow cytometer of the fluorescent intensity on a given cell. The random variable $U$ represents the natural fluorescence of the cell: it is its baseline reaction regardless the impact of the drug administrated during the experiment. The random variable $X$ stands for the quantitative effect of the reagent when the cell actively reacts to the treatment. At last, $A$ represents the effectiveness of the treatment on the cell. We can therefore list below our definitions related to \eqref{eq:model}:

\begin{itemize}
\item $U$ represents the baseline ``noise" on the observations, leading to the convolution model. We assume that the distribution of $U$ is absolutely continuous with respect to the Lebesgue measure  on $\RR$ and the density of $U$ is assumed to be \textbf{known} and is denoted by $g$. This assumption for example translates a preliminary step of calibrating the cytometer before any treatment with a non-parametric estimation of the density $g$.
\item $A$ follows a Bernoulli distribution $A \sim \cB(1-p)$ and $p$ is an \textbf{unknown} parameter. We use the convention $\PP[A=0]=p$ (the cell does not react to the administrated reagent) while $\PP[A=1]=1-p$ (the reagent induces an effect on the cell). 
\item $X$ is the effect of the treatment when the cell is reacting to the treatment and we assume that the distribution of $X$ is also absolutely continuous with respect to the Lebesgue  measure on $\RR$, with an \textbf{unknown} density $f$. We also aim to estimate $f$ to quantify the effect of the treatment.
\end{itemize}

For example, in each situation illustrated in Figure \ref{fig:cyto_exemple}, we can see that  the baseline density $g$ is represented by the blue area while the ``mixture" distribution obtained after the statistical contamination of $U$ by $A X$ is represented by the red area.

Even though relatively simple in appearance, we will see below that Equation \eqref{eq:model} deserves both theoretical development and numerical efforts to obtain a practical and  theoretically supported method of estimation for $p$ and $f$. For example, we could think of the use of a moment estimation strategy to obtain informations on $p$. Unfortunately, we can rapidly compute that
$$
\mathbb{E} (Z) = \mathbb{E}( U) + (1-p) \mathbb{E} (X) \qquad \text{and} \qquad \mathbb{V}(Z) =  \mathbb{V}(U)+ p(1-p) \mathbb{E} (X)^2 + (1-p)\mathbb{V} (X), 
$$
leading to a non trivial relationship between the known moments of the distribution of $U$, the Bernoulli parameter $p$, and the moments of the distribution of $X$. This difficulty has been partially solved with the help of a non-parametric deconvolution strategy in \cite{vanes1,vanes2} adapted to the situation of atomic deconvolution. This estimator is inspired from the seminal contributions of \cite{Carroll} and \cite{Fan}.
In particular,  $\cite{vanes1}$ builds an ``ideal" consistent estimator and proves a CLT related to this estimator while \cite{vanes2} studies the associated convergence rates and their minimax properties. 
Unfortunately, both works do not address the important question of adaptivity to the smoothness class $\Hs$ of the unknown density $f$(see below for details). The choice of the ``good" smoothness parameter $s$ may be a real issue from a practical point of view. In the general settings of non parametric regression, this adaptivity property may be attained by different strategies.

\subsection{Adaptation}

Hence, the method proposed in these works require an important improvement to produce a fully adaptive estimation strategy, and therefore to obtain a tractable algorithm for the estimation of $p$ and $f$. Adaptivity of estimators is a very desirable property: in non-parametric statistics, estimators are importantly affected by a wrong choice of the parameters that are ideally smoothness-dependent. This is in particular the case for the bandwidth of a kernel deconvolution (see \cite{Fan}) or a frequency threshold in a Tychonov-type regularization (see \cite{Tsybakov_book}) for example. Different strategies have been designed these last decades to attain \textit{ad-hoc}  adaptivity of estimators. Among them, model selection theory introduced initially in \cite{BM} produced a lot of interesting derivations in non-parametric and high-dimensional statistics, while the cornerstone of a tractable use of such an approach relies on a suitable minimal penalty calibration. Nevertheless, this calibration is sometimes difficult  both from a theoretical and from a practical point of view and other methods can  also be used to produce efficient adaptive estimators. For example, resampling methods with an additional cross-validation procedure (see \textit{e.g.} \cite{cv}) is a popular method but we should notice that deriving theoretical results with this simple method requires a significant amount of work (see the recent contributions of \cite{arlot1} for density estimation and of  \cite{arlot2} on general resampling methods).
In statistical signal processing, adaptivity of estimation may also be achieved with the help of a suitable decomposition basis and thresholding strategy. This is what is done with wavelets on Besov balls on specific situations: introduced for density estimation in \cite{djkp} and used in different type of inverse problems in \cite{BG10,BG2} for example. Nevertheless, wavelets are not so easy to use with mixture problems  and  we propose to follow another possible guideline. We will use the so-called Lepskii method to derive adaptive estimators. This method introduced in \cite{L92} has been successfully applied in many non-parametric estimation problems and we refer to \cite{LG11},\cite{L15} for recent contributions using this method. We also refer to \cite{Chichignoud} for an introduction of this method in different frameworks. The success of this method relies on a good bias-variance trade-off in the estimation procedure, and does not require a too much involved parameters tuning step.

\subsection{Main assumptions}
 Below, we will use  $\ii$ to refer to the complex number such that $\ii^2=-1$ and the notation ``$:=$" will refer to the definition of a mathematical object. We assume that the random variables $U$ and $X$ have a bounded second moment.
The characteristic function of any random variable $W$ will be denoted by $\Phi_{W}$. Therefore, we will frequently use the following notations:

$$
\forall t \in \RR \qquad 
\PZ(t) := \mathbb{E}[e^{\mathfrak{i} t Z}] \qquad\text{and} \qquad \PU(t) := \mathbb{E}[e^{\mathfrak{i} t U}] 
 \qquad \text{and} \qquad \PX(t) := \mathbb{E}[e^{\mathfrak{i} t X}],
$$
where $(Z,U,X)$ are the random variables involved in Equation (\ref{eq:model}). We assume that $U$ (resp. $X$) has a known density $g$ (resp. unknown density $f$) with respect to the Lebesgue measure. We will use for the sake of convenience the notation $\Phi_h$ to refer to the Fourier transform of any density $h$:
$$
\Phi_h(t) = \int_{\RR} e^{\ii t x} h(x) dx.
$$
The notation $\lesssim$ refers to an inequality up to a multiplicative constant independent of $n$.

\vspace{1em}

It is well known that the Fourier transform of $Z$ and its empirical counterpart may be used to obtain reliable estimations of $\PX$ when dealing with a standard convolution inverse problem (see \cite{Fan}) at the price of an assumption on the deconvolution operator translated in $\PU$. Moreover, let us note that our model includes the situation where $p=0$, which turns our atomic deconvolution problem into the standard deconvolution problem. Therefore, it is expected that our estimation problem can be solved at the minimal price of some standard smoothness and invertibility conditions on the Fourier transform of $U$ and some smoothness assumptions on $f$. Indeed, these kind of assumptions are well known in the inverse problem litterature and commonly used to derive convergence rates of estimators (see among many references the work \textit{e.g.} \cite{FanKoo} where this assumption is used in its great generality, or \cite{Cavalier} where this assumption is specified in the Poisson inverse problem situation).
We also refer to \cite{BG10,BCG} for other applications of this kind of assumptions with different ``deconvolution $+$ mixture problems". We are therefore driven to introduce two sets of densities $\Hs$ and $\Hnu$:

\vspace{0.5em}
\noindent
$\bullet$ The set $\Hs$ denotes the set of densities $f$ that belong to the Sobolev space of regularity $s$ (and radius $R$) described with the help of the associated characteristic functions  $\Phi_f$ such that:
$$
\Hs := \left\{ h \in \mathbb{L}^2(\RR) \, : \,  \int_{\RR} |\Phi_h(t)|^2 (1+|t|^{2s}) dt \leq R^2 \right\}
$$
Below, the density of $X$ denoted by $f$ is assumed to belong to $\Hs$ with an \textit{unknown} parameter $s$, meaning that we assume that
 $ \int_{\RR} |\PX(t)|^2 (1+|t|^{2s}) dt \leq R^2 .$

\vspace{0.5em}
\noindent
$\bullet$ The set $\Hnu$ denotes the set of densities such that the Fourier transform satisfies the smoothness and ``invertibility" condition :
$
 \exists \, (d_0,d_1) \, : 0<d_0<d_1$ and $d_0 |t|^{-\nu} \leq   |\Phi(t)| \leq d_1 |t|^{-\nu}$  as $|t| \longrightarrow + \infty.$

$$
\Hnu := \left\{ h \in \mathbb{L}^2(\RR) \, : \,   \exists \, 0<d_0<d_1  \quad d_0 |t|^{-\nu} \leq   |\Phi_h(t)| \leq d_1 |t|^{-\nu}\right\}.
$$
Below, we will assume that the density $g$ of the random variable $U$ belongs to $\Hnu$.\\

\noindent Let us briefly comment on this last assumption.
In \cite{Fan},  $\Hnu$ refers to the set of ordinary smooth densities of order $\nu$, which includes many distributions as gamma, double exponential distributions. Note that it would be possible to address the super-smooth case with an exponential decrease of the Fourier transform to handle Gaussian or Cauchy densities. Nevertheless, we have chosen in this work to restrict our study to the ordinary smooth case for the sake of brievety.
As pointed in Section \ref{sec:model}, we assume that $g$ is known, meaning that $\PU$ is known on $\RR$. Such an assumption is reasonnable regarding the practical example we want to handle where we can repeat the experiments for the calibration of the cytometer many times.\\

The estimators introduced in \cite{vanes1} exploit the knowledge of $\Hnu$ and of $\Hs$ (which is more anoying for practical purpose) to obtain an optimal estimator with the important assumption of the knowledge of $s$ (see Section \ref{sec:nonadap}).
\subsection{Main results}
In this article we introduce an adaptive procedure to estimate the unknown parameters $p$ and $f$. We consider the estimators $\hat{p}_n$ and $\hat{f}_n$ proposed in \cite{vanes1} based on a kernel estimation using the relationship between the Fourier transform of each random variables because Equation \eqref{eq:model} yields:
\begin{equation}\label{eq:Fourier}
\forall t \in \RR \qquad \PZ(t) = \PU(t)\left[p+(1-p) \PX(t)\right].
\end{equation}
We detail in Section \ref{subsec:non_ada_p} and \ref{subsec:non_ada_f} the non-adaptive construction proposed by \cite{vanes1}. We then develop a well-designed strategy (see \textit{e.g.} \cite{L92}) to obtain an adaptive estimator of $p$ and $f$. This method allows for choosing among a grid of regularity, the associated bandwidth parameter that realizes the bias-variance tradeoff.
The precise construction of the adaptive esitmator $\hpjn$ and $\hfjn$ are given in Section \ref{subsec:lepskip} and \ref{subsec:lepskif}. 
The aim of this article is to establish the consistency rate of our adaptive strategy. More precisely, we will prove the following results. The first result concerns the estimation of $p$.
\begin{theo}[Minimax adaptivity of $p$]\label{theo:lepski_p}
Assume that $f$ belongs to $\Hs$ and $g$ to $\Hnu$ with $\nu > 1$, if $p \in (0,1)$, then the estimator $\hpjn$ defined in \eqref{eq:pnj} and \eqref{def:j_chap_p} satisfies:
$$
\E \left( \hpjn-p \right)^2 \lesssim \left(\log n\right) n^{-(2s+1)/(2s+2\nu)}.
$$
\end{theo}
The proof of Theorem \ref{theo:lepski_p} follows  a multiple-testing strategy jointly used with a concentration inequality. The additional log term (regarding the minimax rate $n^{-(2s+1)/(2s+2\nu)}$) involved in our upper bound is the price to pay for using a multiple testing strategy and identify the smoothness parameter.\medskip

\noindent
Then we use a plug-in strategy to estimate the function $f$ from its Fourier transform. 
\begin{theo}[Minimax adaptivity of $f$]\label{theo:lepski_f}
For any $p\in(0,1)$, assume that $Z$ has a finite variance and assume that $f \in \Hs$ and $g \in \Hnu$  with $\nu>1$ and  select $\hjf$ as in Equation \eqref{eq:defln}. Then, the estimator $\hat{f}_{n,\hjf}$ defined in \eqref{def:h_f_h_j} satisfies:
$$
\E \|\hat{f}_{n,\hjf} - f \|_2 \lesssim 
  (\log n)^{a+1/2} \, n^{-s/(2s+2\nu+1)},
$$
where $a>0$ is defined in Equation \eqref{eq:taun}.
\end{theo}
We emphasize that our estimators produced both for $p$ and $f$ induced by the Lepskii rule are non-asymptotic and fully-adaptive with respect to the smoothness parameter $s$. Therefore, our results produce an adaptive minimax upper bound up to some log term.\medskip

We should at last remark that the previous results require the important knowledge of the convolution operator brought by $\Phi_U$. Even if this assumption is legitimate for our intended applications, this may not be the case in other situations. When the operator described by $\Phi_U$ is unknown and has to be estimated from the data, then the problem falls into the framework of deconvolution with noise in the operator (see \cite{cavalier2}). In such a case, it is highly suspected that the attainable rates may be damaged by the preliminary estimation of $\Phi_U$. Such a theoretical development is beyond the scope of this work and certainly deserves some careful derivations to understand how the noise on the eigenvalues of the ``deconvolution operator" is propagated (see \textit{e.g.} \cite{cavalier3} and \cite{cavalier1}).\medskip

The article is organized as follows: Section \ref{sec:p} details first the non-adaptive and then the adaptive estimation of the contamination parameter $p$ while Section \ref{sec:f} proposes to solve similar estimation problems for $f$. We present numerical simulations in Section \ref{sec:numerics} on both simulated and real data-sets and the interest of our results for biological purposes.

\section{Estimation of the contamination rate $1-p$}
\label{sec:p}
We first consider the problem of estimating the contamination rate $1-p$. We recall the non-adaptive  results obtained by \cite{vanes1}  and expose our adaptive strategy afterwards.
\subsection{Non adaptive estimation of $p$}\label{sec:nonadap}
\label{subsec:non_ada_p}
We describe below the estimators proposed by \cite{vanes1}, that will be used to obtain our adaptive procedure.
A key observation relies on Equation \eqref{eq:Fourier}.
This last equation makes it possible, from $n$ independent observations $(Z_i)_{1 \leq i \leq n}$, to obtain an estimation of $\PX$ and $p$. Given this set of observations, the first  step consists in  the introduction of the empirical estimator of the Fourier transform of $Z$:
$$
\forall t  \in \RR \qquad \hPZ(t) := \frac{1}{n} \sum_{i=1}^n e^{\ii t Z_i}.
$$
A key relationship between $p$ and the several Fourier transforms is given by \eqref{eq:Fourier}.
Since $f$ belongs to $\Hs$, then $\lim_{|t| \longrightarrow + \infty} \PX(t) = 0$. Therefore, we expect to recover $p$ by using the information brought by $\PZ$ and $\PU$ at large frequencies. In particular, Equation \eqref{eq:Fourier} implies the identifiability of the model and provides some insights for an estimation strategy. Second, we may use the knowledge of $\PU$ and in particular the fact that the density of $U$ belongs to $\Hnu$, which entails some lower bounds of $\PU$ for large frequencies. 

We use the construction of  \cite{vanes2}
and introduce a smooth real valued kernel $\PK$ in the Fourier domain that satisfies:
\begin{equation}\label{eq:kernel}
\PK(t) \neq 0 \Longleftrightarrow t \in [-1;1] \qquad \text{and} \qquad \int_{-1}^1 \PK(t) dt = 2,
\end{equation}
and a flatness condition on the neighborhood of $0$ ($\mathcal{V}_k$ below denotes an open neighborhood of $0$):
\begin{equation*}
\forall k \in \RR_+, \quad \exists C_k > 0,\quad \exists \mathcal{V}_k \quad \forall t \in \mathcal{V}_k: \qquad |\PK(t)| \leq C_k |t|^{k} .
\end{equation*}
This last condition is not restrictive and is satisfied  when $\PK$ is chosen for example as:
$$
\PK(t) := \frac{e^{-a |t|^{-m}}}{C_{a,m}} \mathbf{1}_{|t| \leq 1},
$$
for any $a>0$, $m \geq 1$ ($C_{a,m}$ is the normalizing constant associated to Equation \eqref{eq:kernel}). Note that this local condition  may be replaced by a global one because of $\PK$ is bounded on $[-1;1]$. Therefore, we keep the notation $C_k$ and assume that:
\begin{equation}\label{eq:askernel}
\forall k \in \RR_+, \quad \exists C_k > 0 \qquad \forall t \in [-1;1]: \qquad |\PK(t)| \leq C_k |t|^{k}.
\end{equation}
Following the works \cite{Fan,vanes2}, we use the kernel $\PK$ on the Fourier transform $p + (1-p) \PX$ and obtain:
$$
\frac{h}{2} \int_{\RR} \PK(ht)  \frac{  \PZ(t)}{\PU(t)} dt = 
\frac{h}{2} \int_{\RR} \PK(ht)  [p+(1-p) \PX(t)] dt = p + \frac{h (1-p)}{2}\int_{-h^{-1}}^{h^{-1}}\PK(ht) \PX(t) dt.
$$
The last term of the r.h.s. vanishes when $h \longrightarrow 0$. Moreover, it is possible to  obtain a tight upper bound in terms of $h$ of this \textit{bias} term:
\begin{align}
\left| h  \int_{-h^{-1}}^{h^{-1}}\PK(ht) \PX(t) dt \right| &= 
\left| h  \int_{-h^{-1}}^{h^{-1}}\PK(ht)|t|^{-s} \PX(t) |t|^{s} dt \right| \nonumber\\
& \leq  h \sqrt{ \int_{-h^{-1}}^{h^{-1}} |\PK(ht)|^2 |t|^{-2s} dt
\int_{\RR} |\PX(t)|^2 |t|^{2s} dt }  \leq \sqrt{2} C_{s} R h^{s+1/2},\label{eq:biais}
\end{align}
where we applied the Cauchy-Schwarz inequality and then Inequality \eqref{eq:askernel} with $k=s$ and the fact that  $f \in \Hs$.
Therefore, we can write that:
$$
\lim_{h \longrightarrow 0} \frac{h}{2} \int_{\RR} \PK(ht)  \frac{  \PZ(t)}{\PU(t)} dt = p.
$$
We can plug Equation \eqref{eq:Fourier} in the limit above and then define a natural  estimator of $p$ (that will depend on a small bandwith parameter $h_n$):
\begin{equation}\label{eq:defpn}
\hat{p}_n := \frac{h_n}{2} \int_{-h_n^{-1}}^{h_n^{-1}}\PK(h_n t) \frac{\hPZ(t)}{\PU(t)} dt.
\end{equation}
According to \eqref{eq:biais}, we can compute an upper-bound of the bias of $\hat{p}_n$ as:
$$
[\EE \hat{p}_n - p]^2 \leq 2 C_{s}^2 R^2 h_n^{2s+1}.
$$
The variance of $\hat{p}_n$ is handled in a standard way following the arguments of \cite{Fan} with $\nu>1$:
\begin{equation}\label{eq:var}
\VV(\hat{p}_n) \lesssim \frac{h_n^2}{n} \int_{\RR} \left|\int_{\RR}e^{\ii t z} \frac{\PK(h_n t)}{\PU(t)} dt \right|^2 dz \lesssim \frac{h_n}{n} \int_{\RR} \left| \frac{\PK(t)}{\PU(t h_n^{-1})}\right|^2 dt \lesssim \frac{1}{n h_n^{2\nu-1}}.
\end{equation}
Now, a classical optimization of the bias-variance tradeoff yields the optimal (non-adaptive) choice for the bandwidth parameter  $h_n^{\star} := n^{-\frac{1}{2s+2\nu}}.$ It can be shown  that this choice leads to the minimax optimal consistency rate $n^{-(2s+1)/(2s+2\nu)}$. In other words, we have obtained:
\begin{theo}[Theorem 1(i) and 5 in \cite{vanes2}]\label{theo:estimp}
Assume that $f \in \Hs$ and $g \in \Hnu$ with $\nu > 1$, then the choice $h_n = n^{-\frac{1}{2s+2\nu}}$ in \eqref{eq:defpn} leads to an estimator $\hat{p}_n$ that satisfies the consistency rate:
$$
\mathbb{E} [|\hat{p}_n-p|^2] \leq C(s,R) n^{-\frac{(2s+1)}{2s+2\nu}},
$$
where $C(s,R)$ is a positive constant that continuously depend of $s$ and $R$.
Moreover, this estimator is minimax optimal under the additional assumption that $|\phi'(t)|(1 +|t|^\nu)\le d_2$ for all $t\in \R$ and $d_2>0$.
\end{theo}

This result highly depends on the knowledge of $s$ and a bad strategy for the choice of $s$ and $h_n$  will significantly arms the estimation procedure. Nevertheless it is our  starting point to produce an adaptive estimator of $p$.

\subsection{Adaptive estimation of $p$}
\label{subsec:lepskip}

We propose to adapt the Lepskii strategy (see \textit{e.g.} \cite{L92}) to obtain an adaptive estimator of $p$.

\paragraph{Grid on $s$}

We define a grid $\mathcal{S}_n$  and will estimate $s$ with an exploration of the possible values in  $\mathcal{S}_n$. We consider an interval $[0,s_0]$ 
 where we assume that $s\in[0,s_0]$ and define a regularly spaced sequence 
$$s_{max}=s_0>s_1>\dots>s_{k_n}\ge 0,$$ such that for every $0\le j\le k_n$, 
\begin{equation}\label{eq:grid_s}
s_j:=s_0-j\frac{\varepsilon}{\log n}.\end{equation}
In the last formula, $\varepsilon>0$ is a fixed parameter. This parameter will be calibrated later on and will permit to obtain good estimation properties.
For each $0\le j\le k_n$, we associate the bandwidth parameter $h_{n,j}$ that corresponds 
to the optimal bandwidth parameter chosen above when the smoothness of $f$ is known and equal to $s_j$. This rule yields:
\ben
\label{def:hnj}
h_{n,j}:=n^{-1/(2s_j+2\nu)}.
\een
The corresponding estimator is then denoted by:
\begin{equation}\label{eq:pnj}
\hat{p}_{n,j} :=0\vee\left( \frac{h_{n,j}}{2} \int_{-h_{n,j}^{-1}}^{h_{n,j}^{-1}} \PK(h_nt) \frac{\hPZ(t)}{\PU(t)} dt \right) \wedge1.
\end{equation}
We choose to constraint the estimator to obtain a value in $[0,1]$, which will not damage the properties of the estimator since $p\in[0,1]$.

The associated minimax risk in the smoothness class $\mathcal{H}_{s_j}(R)$ is of the order $n^{-(s_j+1/2)/(2 s_j+2\nu)}$ (see Theorem \ref{theo:estimp}). Given a positive parameter $\beta$, we introduce a \textit{penalty} term defined by this minimax risk up to a log term:
\begin{equation}
\label{eq:penalty}
\kappa_{n,j} := \beta \sqrt{\log n} n^{-(s_j+1/2)/(2 s_j+2\nu)},
\end{equation}
We should note the important monotonic variations of the quantities defined above:  when $j>l$ then $s_j<s_l$ and the bandwidth and penalty parameters satisfy $h_{n,j}<h_{n,l}$ and $\kappa_{n,j}>\kappa_{n,l}$. We introduce the notation $s^\star$ for the closest element to the regularity $s$ from below in the grid $\mathcal{G}_n$:
\ben
\label{def:sstar}
s^\star := \sup \left\{ s_j \in \mathcal{G}_n \, : \, s_j \leq s\right\}.
\een
For the sake of convenience, we denote by $j^{\star}$ the integer in $\{0,\dots,k_n\}$ such that  $s^\star=s_{j^\star}$.
This integer deterministically depends on $n$ and $\epsilon$, but we omit this dependence to alleviate the notations. The calibration of the grid yields
$$
0 \leq s - s^\star \leq \frac{\varepsilon}{\log n}.
$$

\paragraph{Model selection with the Lepskii rule}
To construct the adaptive estimator, we use the following decision rule driven by the bias variance decomposition of each estimator of $\hat{p}_{n,j}$:
\begin{equation}
\label{def:j_chap_p}
\hjp:=\inf\bigl\{0\le j\le k_n,\quad \forall l>j\, : \, \quad \left| \hat{p}_{n,j}-\hat{p}_{n,l} \right|<\kappa_{n,l}\bigr\}.
\end{equation}
The decision rule defined by \eqref{def:j_chap_p} is costly from a computational point of view. However, alternative coding strategies inspired from \cite{Katkovnik} make it possible to efficiently implement \eqref{def:j_chap_p}. These strategies rely on the construction of confidence intervals for the sequence $(\hat{p}_{n,j})_{j}$. We will see in Section \ref{sec:numerics} that if such a strategy works for the computation of a scalar paramter as $p$; it is unfortunately no more possible for the estimation of a density.

The objective of the next paragraphs is to establish the consistency rate of this adaptive strategy, whose performance is described by Theorem \ref{theo:lepski_p}. We first detail the Lepskii strategy, than give a concentration result and conclude.


\subsubsection{Analysis of the Lepskii rule}

We begin by the statement of two preliminary results. The first one links the average performance of $\hpjn$ with some deviations inequalities.
\begin{pro}\label{prop:lepski}
Assume that $f \in \Hs$ and $g$ is a known density in $\Hnu$.Then the estimator $\hpjn$ defined with \eqref{eq:pnj} and \eqref{def:j_chap_p} satisfies
 $$
\E \abso{\hpjn-p}^2 \le \left(\beta^2 \log(n) +C\right) e^{2\varepsilon/(2s+2\nu)} n^{-(2s+1)/(2s+2\nu)} + \sum_{l  \ge j^\star} \Pro\Bigl( \abso{\hat{p}_{n,l}-p} \ge \frac{\kappa_{n,l}}{2}\Bigr).
$$
\end{pro}

\begin{proof}
We can decompose the estimation error as

$$\EE\left(\left|\hpjn-p\right|^2\right)=
\EE \left( \left|\hpjn-p\right|^2 \mathbf{1}_{\hjp \leq j^\star} +
\left|\hat{p}_{n,\hjp}-p\right|^2 \mathbf{1}_{\hat{j}_n > j^\star}
 \right) 
$$

\begin{itemize}
\item We first consider the event $\{\hjp \le j^\star\}$ and apply the inequality $(a+b)^2 \leq 2 a^2+2 b^2$ to obtain:
\begin{equation}\label{eq:triangle}
\left|\hpjn-p \right|^2 \le2 \abso{\hpjn-\hpjs}^2+2\abso{\hpjs-p}^2.
\end{equation}
We study the first term of the r.h.s. of \eqref{eq:triangle}.
On the event $\{\hjp < j^\star\}$, the definition \eqref{def:j_chap_p} of $\hjp$ yields:
$$\abso{\hpjn-\hpjs} <\kappa_{n,j^\star},$$
and obviously when $\hjp = j^\star$ the upper bound above also holds. Hence, we deduce that:
$$
\un_{\hjp \le  j^\star} \abso{\hpjn-\hpjs}^2  < \{\kappa_{n,j^\star}\}^2.
$$
Moreover, writing $s^\star=s+(s^\star-s)$ in the definition of the penalty \eqref{eq:penalty} leads to:
\bea
\kappa_{n,j^\star}&=\beta \sqrt{\log(n)} n^{-(s^\star+1/2)/(2s^\star+2\nu)}\\
&=\beta\sqrt{\log(n)} n^{-(s+1/2)/(2s+2\nu)}\exp\left[\left(\frac{s+1/2}{(2s+2\nu)}-\frac{s^\star+1/2}{(2s^\star+2\nu)}\right)\log n \right]\\
&\le  \beta\sqrt{\log(n)} n^{-(s+1/2)/(2s+2\nu)}\exp\left[\frac{2 \nu \epsilon \{\log n\}^{-1}}{(2s^\star+2\nu)(2s+2\nu)} \log n \right]\\
&\le \beta \exp\Bigl[\frac{\varepsilon}{2s+2\nu }\Bigr] \sqrt{\log(n)} n^{-(s+1/2)/(2s+2\nu)}.\\
\eea
We can handle the second term of the r.h.s.  of \eqref{eq:triangle} easily: the nonadaptive estimator obtained with $j^\star$  satisfies the upper bound obtained in Theorem \ref{theo:estimp}, so that:
\bea
\E(\abso{\hat{p}_{n,j^\star}-p}^2)&\le C n^{-(2s^\star+1)/(2s^\star+2\nu)}\\
&\le C \exp\left(\frac{2\varepsilon}{2s+2\nu}\right)n^{-(2s+1)/(2s+2\nu)},
\eea
where $C$ is a positive constant independent from $s$, $n$, and $\varepsilon$ (the constant $C$ only depends on $R$ and $s_0$).
Then, 
\ben
\label{eq:cas1}
\E(\un_{\hjp \le  j^\star} \abso{\hpjn-p}^2)\le  \left(\beta^2 \log(n) +C\right) \exp\left(\frac{2\varepsilon}{2s+2\nu}\right) n^{-(2s+1)/(2s+2\nu)}.
\een
\item We consider now the complementary event $\{\hjp > j^\star\}$ and first remark that
$$
\hjp >  j^\star \Longleftrightarrow \forall j \leq j^\star\, \exists \, l > j   \, : \, \abso{\hat{p}_{n,j} - \hat{p}_{n,l}} \ge \kappa_{n,l}.
$$
Therefore, we deduce that
\begin{align}
\{\hjp> j^\star\} &=\bigcap_{j \leq j^\star} \bigcup_{l>j}\Bigl\{ \abso{\hat{p}_{n,j} - \hat{p}_{n,l}} \ge \kappa_{n,l} \Bigr\} \nonumber\\
&\subset  \bigcup_{l>j^\star}\Bigl\{ \abso{\hat{p}_{n,j^\star} - \hat{p}_{n,l}} \ge \kappa_{n,l} \Bigr\} \nonumber\\
& \subset \bigcup_{l>j^\star}\Bigl\{ \abso{ \hat{p}_{n,l} - p} \ge \frac{\kappa_{n,l}}{2} \Bigr\} \bigcup_{l>j^\star}\Bigl\{ \abso{\hat{p}_{n,j^\star} - p} \ge \frac{\kappa_{n,l}}{2} \Bigr\}.  \label{eq:inclusion}
\end{align}
where the second line comes from the inclusion $\displaystyle\cup_{j \leq j^\star}\{. \} \subset \displaystyle\cup_{j = j^\star} \{. \}$ and the last line from the triangle inequality.
We have seen that the map $l\mapsto  \kappa_{n,l} $ is increasing then when $l>j^\star$, we have $\kappa_{n,l} \geq \kappa_{n,j^\star}$ so that
$$\bigcup_{l > j^\star} \Bigl\{ \abso{\hat{p}_{n,j^\star}-p} \ge \frac{\kappa_{n,l}}{2}\Bigr\}\subset\Bigl\{ \abso{\hat{p}_{n,j^\star}-p} \ge \frac{\kappa_{n,j^\star}}{2} \Bigr\}.
$$
Using this last inclusion in the second union of the right hand side of \eqref{eq:inclusion}, we finally deduce that
\begin{equation*}
\{\hjp> j^\star\} \subset  \bigcup_{l \geq j^\star}
\Bigl\{ \abso{\hat{p}_{n,l} - p} \ge \frac{\kappa_{n,l}}{2} \Bigr\}.
\end{equation*}
From the obvious upper-bound $\abso{\hpjn-p}^2\le1$,  we can conclude that:
\ben
\label{eq:eq2}
\E\bigl(\abso{\hpjn-p}^2\un_{\hat{j}_n> j^\star}\bigr)
\le \E\bigl(\un_{\hjp> j^\star}\bigr) \le 
\sum_{l \ge j^\star} \Pro\Bigl( \abso{\hat{p}_{n,l}-p} \ge \frac{\kappa_{n,l}}{2}\Bigr).
\een
\end{itemize}
Gathering \eqref{eq:cas1} and \eqref{eq:eq2} leads to the conclusion.
\hfill $\square$
\end{proof}

\subsubsection{Concentration inequality}
\label{subsec:concentration_p}
The adaptivity property of the estimator $\hpjn$ will be deduced from the multiple-testing strategy induced by Proposition \ref{prop:concentration_f}. A baseline property to successively apply this approach will be a derivation of a concentration inequality on $\hat{p}_{n,l} - p$ for each fixed $l$. 
Next proposition states that such a concentration inequality holds.

\begin{pro}\label{prop:concentration_p}
Let $f \in \Hs$, $g$ a known density in $\Hnu$ with $\nu>1/2$, and $l \ge j^\star$, then
$$
\Pro\Bigl( \abso{\hat{p}_{n,l}-p} \ge \frac{\kappa_{n,l}}{2}\Bigr) \le 2 n^{-\beta^2/64}.
$$
\end{pro}
\begin{proof}The  proof is divided into two steps. 

\noindent
\textit{\underline{Step 1: reduction to a concentration inequality.}}
For any $l\ge j^\star$ and $n \in \NN$, we write
\bea
\hat{p}_{n,l}-p& = \frac{h_{n,l}}{2} \int_{-h_{n,l}^{-1}}^{h_{n,l}} \PK(h_{n,l} t) \frac{\hPZ(t)}{\PU(t)} dt -p \\
& =  \frac{h_{n,l}}{2} \int_{-h_{n,l}^{-1}}^{h_{n,l}} \PK(h_{n,l} t) \frac{\hPZ(t)-\PZ(t)}{\PU(t)} dt  +\frac{h_{n,l}^{-1}}{2} \int_{-h_{n,l}^{-1}}^{h_{n,l}} \PK(h_{n,l} t) \frac{\PZ(t)}{\PU(t)} dt -p \\
& =   \frac{h_{n,l}}{2} \int_{-h_{n,l}^{-1}}^{h_{n,l}} \PK(h_{n,l} t) \frac{\hPZ(t)-\PZ(t)}{\PU(t)} dt  +  \frac{(1-p) h_{n,l}^{-1}}{2} \int_{-h_{n,l}^{-1}}^{h_{n,l}} \PK(h_{n,l} t) \PX(t) dt.  \\
\eea
The triangle inequality and the bias upper bound provided by Equation \eqref{eq:biais} yield:
$$
\abso{\hat{p}_{n,l}-p} \leq \left|\frac{h_{n,l}}{2} \int_{-h_{n,l}^{-1}}^{h_{n,l}} \PK(h_{n,l} t) \frac{\hPZ(t)-\PZ(t)}{\PU(t)} dt \right| + C_{s} R h_{n,l}^{s+1/2}.
$$
The r.h.s. of the upper bound is of the order $h_{n,l}^{s+1/2} = n^{-(s+1/2)/(2s_l+2\nu)}$ while
$\kappa_{n,l} = \beta \sqrt{\log n} n^{-(s_l+1/2)/(2s_l+2\nu)}$. From the definition of the grid on $s$, $l\ge j^\star$ implies that $s_l\le s_{j^\star}$ so that $s_l \le s$. Therefore, we can write:
$$
h_{n,l}^{s+1/2} \{\kappa_{n,l}\}^{-1} \lesssim \{\log n\}^{-1/2} n^{(s_l-s)/(2s_l+2\nu)} \lesssim \frac{1}{\sqrt{\log n}}.
$$
We then deduce that a sufficiently large $n$ exists such that for $l \ge j^\star$, we have $C_s R h_{n,l}^{s+1/2} < \kappa_{n,l}/4$ and
\begin{eqnarray}
\Pro\Bigl( \abso{\hat{p}_{n,l}-p} \ge \frac{\kappa_{n,l}}{2}\Bigr) &\leq &
\Pro\left( \left| \frac{h_{n,l}}{2} \int_{-h_{n,l}^{-1}}^{h_{n,l}^{-1}} \PK(h_{n,l} t) \frac{\hPZ(t)-\PZ(t)}{\PU(t)} dt \right|  \ge \frac{\kappa_{n,l}}{4}\right) \nonumber\\ 
& \leq & \Pro\left( \left|\frac{1}{n} \underbrace{\sum_{k=1}^n \frac{h_{n,l}}{2} \int_{-h_{n,l}^{-1}}^{h_{n,l}^{-1}} \PK(h_{n,l} t) \frac{e^{\ii Z_k t} -\PZ(t)}{\PU(t)} dt}_{:=\xi_{k,l}} \right|  \ge \frac{\kappa_{n,l}}{4}\right).\label{eq:step1}
\end{eqnarray}

\noindent
\textit{\underline{Step 2: application of the Bernstein concentration inequality.}}
To handle the r.h.s. of Inequality \eqref{eq:step1}, we will apply the Bernstein inequality (see Theorem 2.9 of \cite{BLM} and exercice 2.8 therein, a precise statement is given in Theorem \ref{theo:bernstein} of our Appendix \ref{sec:appendix}) to the random variables
$$
\xi_{k,l} = \frac{h_{n,l}}{2} \int_{-h_{n,l}^{-1}}^{h_{n,l}^{-1}} \PK(h_{n,l} t) \frac{e^{\ii Z_k t} -\PZ(t)}{\PU(t)} dt.
$$
We know that the i.i.d. random variables $(\xi_{k,l})_{1 \leq k \leq n}$  are centered and satisfy an almost sure bound:
$$
\forall k \in \{1 \ldots n\} \qquad |\xi_{k,l}| \leq \frac{h_{n,l}}{2} \int_{-h_{n,l}^{-1}}^{h_{n,l}^{-1}} \frac{|\PK(h_{n,l} t)|}{|\PU(t)|} dt.
$$
We now use the assumption that $g \in \Hnu$ to deduce that a $d_0$ exists such that $\PU(t) (1+|t|)^{\nu} \geq d_0$ for all $t \in \RR$. Therefore, a large enough constant $C$ exists such that  for any $k \in \{1 \ldots n\}$:
\bea
 |\xi_{k,l}| & \leq  \frac{h_{n,l}}{2} \int_{-h_{n,l}^{-1}}^{h_{n,l}^{-1}} \frac{|\PK(h_{n,l} t)|(1+|t|)^{\nu}}{|\PU(t)|(1+|t|)^{\nu}} dt\\
 & \leq
  \frac{h_{n,l}}{2 d_0} \int_{-h_{n,l}^{-1}}^{h_{n,l}^{-1}} |\PK(h_{n,l} t)|(1+|t|)^{\nu}dt\\
  & \leq (1+|h_{n,l}^{-1}|)^{\nu}\frac{h_{n,l}}{2 d_0} \int_{-h_{n,l}^{-1}}^{h_{n,l}^{-1}} |\PK(h_{n,l} t)|dt \leq  C h_{n,l}^{- \nu}.
\eea
We can also bound the variance of each $\xi_{k,l}$ and a computation similar to Equation \eqref{eq:var} leads to $\VV(\xi_{k,l}) \leq h_{n,l}^{-2\nu+1}.$ We then apply Theorem \ref{theo:bernstein} with $b=C h_{n,l}^{-\nu}$ and $v= n h_{n,l}^{-2\nu+1}$ and deduce that:
$$
\Pro\left( \left|\frac{1}{n} \sum_{k=1}^n \xi_{k,l}\right| \geq \frac{\kappa_{n,l}}{4} \right) \leq 2 \exp \left( - \frac{n \kappa_{n,l}^2}{32 \left( h_{n,l}^{-2\nu+1}+ C h_{n,l}^{-\nu} \kappa_{n,l}/3)\right)}.
\right)$$
Now, remark that $h_{n,l}^{-2\nu+1} = n^{(2\nu-1)/(2\nu+2s_l)}$ while $h_{n,l}^{-\nu} \kappa_{n,l} =\beta n^{\nu/(2\nu+2s_l)}n^{-(s_l+1/2)/(2\nu+2s_l)} \sqrt{\log n}$. Since we assumed that $\nu>1/2$, the main contribution in the denominator of the exponential bound above is brought by $h_{n,l}^{-2\nu+1}$. We therefore conclude that a large enough $n$ exists such that:
$$
\Pro\left( \left|\frac{1}{n} \sum_{k=1}^n \xi_{k,l}\right| \geq
 \frac{\kappa_{n,l}}{4} \right) \leq 2 \exp \left( - \frac{n \kappa_{n,l}^2}{64  h_{n,l}^{-2\nu+1}} \right) =2 n^{-\beta^2/64},
$$
because the penalty term $\kappa_{n,l}$ has been designed so that 
$\frac{n \kappa_{n,l}^2}{h_n^{-2 \nu+1}} = \beta^2 \log(n)$.
$\hfill  \square$
 \end{proof}

\subsubsection{Proof of Theorem \ref{theo:lepski_p}}
We can achieve the proof of Theorem \ref{theo:lepski_p} and establish the adaptation of our estimator $\hpjn$: Propositions \ref{prop:lepski} and \ref{prop:concentration_p} yield
$$
\E \abso{\hpjn-p}^2 \le\left(\beta^2 \log(n) +C\right) \exp\left(\frac{2\varepsilon}{2s+2\nu}\right) n^{-(2s+1)/(2s+2\nu)} +2  \sum_{l \ge j^\star} n^{-\beta^2/64}.
$$
The size of the grid is bounded from above by $\varepsilon^{-1} s_{max} \log n  $.
Therefore, for $n$ large enough
$$
\E \abso{\hpjn-p}^2 \le(\beta^2 +1)e^{\frac{2\varepsilon}{2s+2\nu}}(\log n)  n^{-(2s+1)/(2s+2\nu)} + \varepsilon^{-1} s_{max} \log n \,  n^{-\beta^2/64}.
$$
It remains to choose the constant $\varepsilon$ and $\beta$. We first pick $\beta$ such that 
\ben
\label{eq:calibr_beta}n^{-\beta^2/64}\le  n^{-(2s+1)/(2s+2\nu)},  
\een
then we choose $\varepsilon$ in order to minimize the constant, regardless the value of $s$,
\ben
\label{eq:calibr_eps}(\beta^2 +1)e^{\frac{2\varepsilon}{2s+2\nu}}+ s_{max} \varepsilon^{-1}. 
\een
These choices yield the desired adaptive property (optimal up to a log term):
$$
\E \abso{\hpjn-p}^2  \lesssim  \left(\log n \right) n^{-(2s+1)/(2s+2\nu)}.
$$

$\hfill  \square$
\begin{rem}
For numerical experiments, the calibration of the parameters $\varepsilon$ and $\beta$ given by \eqref{eq:calibr_beta} and \eqref{eq:calibr_eps} is crucial. Since the value of $s$ is unknown, these inequality have to be true for all $s\in[0,s_{max}]=[0,s_{0}]$. We are driven to an optimal choice of the form
\bea
\beta^\star&=8 \sqrt{ \frac{2s_{max}+1}{2s_{max}+2\nu}} \, \text{and} \, 
\varepsilon^\star=\text{argmin}\left\{((\beta^\star)^2 +1)e^{\frac{\varepsilon}{\nu}}+ s_{max} \varepsilon^{-1} \right\}.
\eea
\end{rem}

\section{Estimation of $f$} 
\label{sec:f}
We keep the same presentation and first briefly describe  the estimator proposed by \cite{vanes2}. We describe our adaptive procedure.
\subsection{Non adaptive approach}
\label{subsec:non_ada_f}
The estimation of $f$ is highly similar to the one of $p$ and relies on the Fourier transform of $U$ and of the empirical data. It exploits the relationship:
$$
\PX(t) = \frac{\PZ(t) - p\PU(t)}{(1-p)\PU(t)},
$$
and therefore uses a plug-in estimator of $p$ and of $\PZ$ with another kernel smoothing strategy. Following \cite{vanes2}, we introduce a second kernel $Q$ whose Fourier transform is denoted by $\PQ$ has its support in $[-1,1]$ and satisfies the following assumptions:
\begin{equation}\label{eq:kernel_f}
\PQ(0)=1, \qquad \forall t \in \RR \qquad |\PQ(t)-1| \leq M |t|^{s}, \qquad \int_{-1}^1|\PQ(t)|^2dt<\infty,
\end{equation}
where $M>0$.
It is worth saying that the choice of the kernel used for this deconvolution is important regarding the numerical results as well as the theoretical ones. For our purpose, we only handle ordinary smooth inverse problem in Sobolev spaces. In that case, it will be enough to handle a very simple kernel $\PQ$, given by the sinc function:
$$
\PQ(t) := \mathbf{1}_{[-1;1]} \qquad \text{so that} \qquad \forall x \in \RR \quad Q(x) = \frac{\sin x}{\pi x} \quad \text{with} \quad Q(0)=\frac{1}{\pi}.
$$
It is immediate to check that this kernel automatically matches the requirement given in Equation \eqref{eq:kernel_f}
 on the smoothness of $\PQ$ around $0$.
We refer to \cite{Delaigle} for a detailed discussion on the influence of the kernel choice from a numerical point of view, and to \cite{CL13} and \cite{LG11} for deeper  insights on functional spaces where the estimation is done (\textit{e.g.} in anisotropic Nikol'skii classes).
The plug-in strategy proposed in \cite{vanes2} reads as follows:
\begin{equation}\label{eq:defestim}
\hat{\Phi}_{X,n} = \frac{\hPZ(t) - \cpn \PU(t)}{(1-\cpn) \PU(t)} \PQ(\delta_n t),
\end{equation}
where $\cpn$ is a preliminary truncated estimation of $p$ that is plugged into Equation \ref{eq:defestim}, with the following rule:
\begin{equation}\label{eq:cpn}
\cpn:= \hat{p}_n \wedge 1-\tau_n.
\end{equation}
The truncation step under $1-\tau_n$ makes it possible to avoid numerical instability when $\hat{p}_n$ is estimated close to $1$. 
It is important to remark that for a theoretical purpose, it is necessary to build a preliminary estimator of $p$ with $\cpn$  independent from the estimated $\PZ$ in \eqref{eq:defestim}. This can be achieved using $\lfloor n/2 \rfloor$ samples to compute $\cpn$ and the remaining ones to build $\hat{\Phi}_{X,n}$ according to our plug-in strategy.
We can now state the minimax consistency result of \cite{vanes2}, which is obtained with a similar bias-variance tradeoff.

\begin{theo}[Theorem 2(i) and 3(i) in \cite{vanes2}]\label{theo:estimf}
Assume $f\in\Hs$ and $g\in\Hnu$ with $\nu > 1$ and set $\tau_n=\log(3n)^{-1}$. Then the choice $h_n = n^{-\frac{1}{2s+2\nu}}$ in \eqref{eq:defpn} and $\delta_n = n^{-1/(2s+2\nu+1)}$ in \eqref{eq:defestim} leads to an estimator $\hat{f}_n$ that satisfies the consistency rate:
$$
\mathbb{E} [\|\hat{f}_n-f\|^2] \lesssim n^{-\frac{(2s)}{2s+2\nu+1}}.
$$
\end{theo}

It is shown in \cite{vanes2} that such estimator is minimax optimal  under the additional assumptions that $s\ge 1/2$ and $|\phi'(t)|(1 +|t|^\nu)\le d_2$ for all $t\in \R$ and $d_2>0$.
In particular, we note that the consistency rate corresponds to the standard rate of ordinary smooth deconvolution inverse problem (see \textit{e.g.} \cite{Fan}).
Again, we stress the fact that the proposed estimator highly depends on $s$ through $\delta_n= n^{-1/(2s+2\nu+1)}$, which is unknown in practice and should be estimated. The next paragraph describes our proposed strategy.

\subsection{Adaptive estimation of $f$}
\label{subsec:lepskif}
\paragraph{Grid on $s$}
The adaptation follows the same strategy as the one for $p$, even though significantly harder from a theoretical point of view.
We still use the same grid  $\mathcal{S}_n$:
$s_0>s_1>\dots>s_{k_n}\ge 0,$ such that for every $0\le j\le k_n$, 
$s_j=s_0-j\frac{\varepsilon}{\log n},$
and the notation $h_{n,j}$ for the bandwidth parameter associated to the optimal estimation of $p$ in $\Hs$ when $s=s_j$, \textit{i.e.} we use
$
h_{n,j} := n^{-1/(2s_j+2\nu)}.
$ 
A second bandwidth parameter associated to the estimation of $f$ is defined by $\delta_{n,j}$, $\forall 0\le j\le k_n$:
\begin{equation}
\label{eq:def_deltanl}
\delta_{n,j} := n^{-1/(2s_j+2\nu+1)} .
\end{equation}

With these two bandwidth parameters derived from a regularity $s_j$, we build a preliminary estimator of $p$ with the observations $(Z_i)_{1 \leq i \leq \lfloor n/2 \rfloor}$:

\begin{equation}\label{eq:pnj1}
\hat{p}^{(1)}_{n,j} = 0\vee\left\{  \left\{ \frac{h_{n,j}}{2} \int_{-h_{n,j}^{-1}}^{h_{n,j}^{-1}} \PK(h_{n,j}t) \frac{\hPZ^{(1)}(t)}{\PU(t)} dt \right\}\wedge 1-\tau_n\right\}
\end{equation}
where $\hPZ^{(1)}(t)$ and $\hPZ^{(2)}(t)$ are two unbiased \textit{independent} estimates of $\PZ(t)$:
$$
\forall t \in \RR \qquad \hPZ^{(1)}(t) := \frac{1}{\lfloor n/2\rfloor} \sum_{k=1}^{\lfloor n/2\rfloor} e^{\ii t Z_k} \qquad \text{and} \qquad 
 \hPZ^{(2)}(t) := \frac{1}{n-1-\lfloor n/2\rfloor} \sum_{k=\lfloor n/2\rfloor+1}^{n} e^{\ii t Z_k}.
$$
Theorem \ref{theo:estimf} entails for a  good choice of the truncation parameter introduced in Equation \eqref{eq:cpn}
and we choose
\begin{equation}\label{eq:taun}
\tau_n:=\log(n)^{-a},
\end{equation} for $a>0$.
Then, we use $\hat{p}^{(1)}_{n,j}$ to obtain an estimator of the empirical Fourier transforms of $X$ by:
\begin{equation}\label{eq:hatf}
\hPXnj :=\frac{ \hPZ^{(2)}(t) - \hat{p}^{(1)}_{n,j} \PU(t)}{(1-\hat{p}^{(1)}_{n,j}) \PU(t)} \PQ(\dnj t)
\end{equation}
The application of the Lepskii method relies on a penalty term, which is defined as
\begin{equation}\label{eq:penalty2}
\rho_{n,l} := \beta (\log n)^{a+1/2} n^{-s_l/(2s_l+2\nu+1)}.
\end{equation}
This penalty is the minimax risk of estimation of $f$ in $\Hs$ when $s_l=s$, up to a log term (see Theorem \ref{theo:estimf}). We recall that for the estimation of $p$, the penalty was associated with a supplementary $(\log n)^{1/2}$ term. However, such a penalty is not strong enough to obtain a good concentration inequality for the estimator of $f$ (see Proposition \ref{prop:concentration_f} below).

\paragraph{Lepski's rule for the estimation of $f$}
We use the penalization introduced in \eqref{eq:penalty2} to define:
\begin{equation*}
\hjf := \inf \{ 0 \leq j \leq k_n , \quad \forall l > j\, :\quad \|\hat{f}_{n,j}-\hat{f}_{n,l}\|_{_2} < \rho_{n,l} \}.
\end{equation*}
Thanks to the Parseval identity, this can be also formulated in terms of the Fourier transforms:
\begin{equation}\label{eq:defln}
\hjf := \inf \{ 0 \leq j\leq k_n ,\quad\forall l > j\, :\quad \|\hPXnj -\hPXnl  \|_{_2} < \rho_{n,l} \}.
\end{equation}
The estimator derived from the penalization above is then written as $\hPXn{\hjf}$ where we used the selection of  $\hjf$ in \eqref{eq:defln} with the definition of $\hPXn{\hjf}$ given in \eqref{eq:hatf}. It leads to the estimation of $f$ itself using the Fourier reconstruction formula:
\begin{equation}
\label{def:h_f_h_j}
\hfjn(x) := \frac{1}{2\pi} \int_{\RR} e^{- \ii t x} \hPXn{\hjf}(t) dt.
\end{equation}

Again, we start with the statement of a proposition that links the average performance of $\hPXn{\hjf}$ with a family of  deviation inequalities.
\begin{pro}\label{prop:lepski2}
Let $f \in \Hs$ and $g$ a known density in $\Hnu$, then  the estimator $\hfjn$ satisfies:
$$
\E \|\hfjn - f\|_{_2} \lesssim (\log n)^{a+1/2} n^{-s/(2s+2\nu+1)} + \sqrt{ \sum_{l \ge j^{\star}} \Pro\Bigl(\|\hat{f}_{n,l} - f\|_{_2}   \ge \frac{\rho_{n,l}}{2}} \Bigr).
$$
\end{pro}
\begin{proof}
The proof is close to the one of Proposition \ref{prop:lepski}. We denote by $s^\star=s_{j^\star}$ the closest element from $s$ from below in the grid $\mathcal{S}_n$ (see Equation \eqref{def:sstar}). Following the same guidelines, we obtain that:

$$
\EE \|\hfjn - f\|_{_2} =
\EE  \left( \|\hfjn - f\|_{_2} \mathbf{1}_{\hjf \leq j^{\star}} +
\|\hfjn - f\|_{_2} \mathbf{1}_{\hjf > j^{\star}} \right) .
$$
\underline{On the event $\{\hjf \leq l^\star\}$:} 
we first apply the triangle inequality
$$\E(\|\hfjn - f\|_{_2} \mathbf{1}_{\hjf \leq j^{\star}}) 
\le\E \left[\left( \|\hfjn-\hat{f}_{j^\star}\|_{_2}+\|\hat{f}_{j^\star}-f\|_{_2}\right)  \mathbf{1}_{\hjf \leq j^{\star}}\right]$$
On ${\hjf \leq j^{\star}}$, the first term is upper bounded by
$$\begin{aligned}
\rho_{n,j^\star}
&=\beta (\log n)^{a+1/2} n^{-s_{j^\star}/(2s_{j^\star}+2\nu+1)}\\
&=\beta (\log n)^{a+1/2} n^{-s/(2s+2\nu+1))} \exp\Bigl[\bigl( \frac{s}{2s+2\nu+1}-\frac{s_{j^\star}}{2s_{j^\star}+2\nu+1}\bigr)\log n \Bigr]\\
&\le \beta (\log n)^{a+1/2} n^{-s/(2s+2\nu+1))} \exp\Bigl( \frac{\varepsilon}{2s+2\nu+1}\Bigr).
\end{aligned}$$
For the second term, we apply the results of the non adaptive estimator in Theorem \ref{theo:estimf}:
$$\E\left(\|\hat{f}_{j^\star}-f\|_{_2}  \mathbf{1}_{\hjf \leq j^{\star}}\right)\le Cn^{-s_{j^\star}/(2s_{j^\star}+2\nu+1)},
$$ and deduce that:
\bean\label{eq:cas12}
\E(\|\hfjn - f\|_{_2} \mathbf{1}_{\hjf \leq j^{\star}}) 
&\le \left(\beta (\log n)^{a+1/2} +C\right) n^{-s/(2s+2\nu+1))} \exp\Bigl( \frac{\varepsilon}{2s+2\nu+1}\Bigr).
\eean
where $C$ is a positive constant independent from $\beta$, $\varepsilon$ and $n$.
\medskip

\noindent
\underline{On the event $\{\hjf > l^\star\}$:} 
On this event, the triangle inequality yields $\|\hat{f}_{n,l} - \hat{f}_{n,j}\|_{_2} \leq \|\hat{f}_{n,l} - f\|_{_2}+\|\hat{f}_{n,j} - f\|_{_2}$,  so that 
$$
\left\{\hjf > j^{\star}\right\} \subset \bigcup_{l \geq j^{\star}}\Bigl\{   \|\hat{f}_{n,l} - f\|_{_2}   \ge \frac{\rho_{n,l}}{2} \Bigr\}.
$$
Therefore, the Cauchy-Schwarz inequality and the inclusion above yields
\begin{eqnarray*}
\E(  \|\hat{f}_{n,\hjf} - f\|_{_2} \mathbf{1}_{\hjf > j^\star} )& \leq & 
\sqrt{\E(  \|\hat{f}_{n,\hjf} - f\|_{_2}^2)} \sqrt{ \E( \mathbf{1}_{\hjf > j^\star} )} \\
& \leq & \sqrt{\sum_{j=0}^{k_n} \E \left(  \|\hat{f}_{n,j} - f\|_{_2}^2\right)}\sqrt{ \E \left( \mathbf{1}_{\bigcup_{l \geq j^\star}\Bigl\{   \|\hat{f}_{n,l} - f\|_{_2}   \ge \frac{\rho_{n,l}}{2} \Bigr\}} \right)} 
\end{eqnarray*}
We now use a refinement of Theorem \ref{theo:estimf} given by Proposition \ref{prop:mieux} in the appendix Section \ref{sec:appendix} to bound $\E( \|\hat{f}_{n,j} - f\|_{_2}^2)$:
\begin{eqnarray*}
\E(  \|\hat{f}_{n,\hjf} - f\|_{_2} \mathbf{1}_{\hjf > j^\star} )& \leq &
\sqrt{\sum_{j=0}^{k_n} \Phi(s,R) n^{-2\min(s_j,s)/(2s_j+2\nu+1)} } \sqrt{ \E \left( \mathbf{1}_{\bigcup_{l \geq j^\star}\Bigl\{   \|\hat{f}_{n,l} - f\|_{_2}   \ge \frac{\rho_{n,l}}{2} \Bigr\}} \right)} \\
\end{eqnarray*}
 Since $k_n \leq s_0 \epsilon^{-1} \log(n)$ and $\Phi$ is continuous in $s$, then $C=\{\sup_{s\in[0,s_0]}\Phi(s,R)\}^{1/2}
 \sqrt{s_0\varepsilon^{-1}} $ is a finite constant. 
Setting now $\alpha= s_0 \epsilon^{-1}$, we shall now study the next sum:
$$
\sum_{j=0}^{k_n} n^{-s_j/(s_j+\nu+1/2)} = \sum_{j=0}^{\alpha \log(n)} n^{-\frac{j \epsilon / \log(n)}{j \epsilon/\log(n)+\nu+1/2}}.
$$
This sum may be divided into two parts: one with $s_j \leq 1$ and the other part with smoothness $s_j \geq 1$. 
If $j \epsilon / \log(n) \geq 1$, then we use the rough bound $s_j/(s_j+\nu+1/2) \geq 1/(\nu+3/2)$ and obtain that:
$$
\sum_{j=\epsilon^{-1} \log(n)}^{k_n} n^{-s_j/(s_j+\nu+1/2)} \leq \alpha \log(n) n^{-1/(\nu+3/2)},
$$
which is a bounded sequence.
Now, if $j \epsilon / \log(n) \leq 1$, we have:
$$
\sum_{j=0}^{\epsilon^{-1}\log(n)} n^{-\frac{j \epsilon / \log(n)}{j \epsilon/\log(n)+\nu+1/2}} \leq 
\sum_{j=0}^{\epsilon^{-1}\log(n)} e^{-\frac{j \epsilon  }{j \epsilon/\log(n)+\nu+1/2}} \leq 
\sum_{j=0}^{\epsilon^{-1}\log(n)} e^{-\frac{ \epsilon  }{\nu+3/2} j} \leq \frac{1}{1-e^{-\frac{\epsilon}{3/2+\nu}}}.
$$
Therefore, a constant $C(\epsilon,s_0,R)$ exists such that:
 
\begin{equation}\label{eq:cas22}
\E(  \|\hat{f}_{n,\hjf} - f\|_{_2} \mathbf{1}_{\hjf > j^\star} ) \leq C(\epsilon,s_0,R)  \sqrt{\sum_{l \geq j^\star} \PP \left( \|\hat{f}_{n,l} - f\|_{_2}   \ge \frac{\rho_{n,l}}{2} \right) }.
\end{equation}
The conclusion follows from \eqref{eq:cas12} and \eqref{eq:cas22}.
\hfill $\square$
\end{proof}

\subsubsection{Bias upper bound}

Once again, the main difficulty of the proof is to obtain a deviation inequality on each event:
$$\Omega_{n,l} := \left\{ \|\hat{f}_{n,l} - f\|_{_2}   \ge \frac{\rho_{n,l}}{2}\right\} = \left\{ \|
\hPXn{l} - \PX\|_{_{2}}^2  \ge \frac{\rho_{n,l}^2}{4}\right\}.
$$
where the last equality comes from the Parseval identity.
We begin by the following statement:
\begin{pro}\label{prop:inclusion}
Assume $l \geq j^\star$, for $n$ large enough one has
$$
\Omega_{n,l} \subset \left\{   \|\hPXn{l} - \EE \hPXn{l}\|_{_{2}}^2  \ge \frac{\rho_{n,l}^2}{8}\right\}.
$$
\end{pro}

\begin{proof} The main ingredient of the proof uses the bias-variance decomposition of  the squared $\LL^2$ norm. Hence, we define
\begin{equation}\label{eq:biasf}
T_1 := \| \E\hPXn{l} - \PX\|_{_{2}}^2  \qquad \text{and} \qquad T_2 := \mathbb{V}(\hPXn{l}) = \E \| \E(\hPXn{l})  -\hPXn{l} \|_{_{2}}^2.
\end{equation}
A key remark comes from the fact that $\hPZ^{(2)}(t)$ is built with a sample of $n-\lfloor n/2 \rfloor $ observations, that are independent of the one used to estimate $p$ with $\hat{p}_{n,l}^{(1)}$. We can write that:
\begin{eqnarray*}
\EE \hPXn{l}(t) 
& = &   \EE \left[\frac{p \PU(t) + (1-p) \PU(t) \PX(t) -\hat{p}_{n,l}^{(1)} \PU(t)}{(1-\hat{p}_{n,l}^{(1)}) \PU(t)}\PQ(\dnl t)  \right] \\
& = &   \EE \left[\frac{ p-\hat{p}_{n,l}^{(1)}}{1-\hat{p}_{n,l}^{(1)}}\right] \PQ(\dnl t) +  \left[ \EE \left[\frac{1-p}{1-\hat{p}_{n,l}^{(1)}}\right]   \PQ(\dnl t) -  1 \right] \PX(t) + \PX(t)  \\
& = & \EE \left[\frac{ p-\hat{p}_{n,l}^{(1)}}{1-\hat{p}_{n,l}^{(1)}}\right] \PQ(\dnl t) (1-\PX(t))+ [\PQ(\dnl t)-1] \PX(t) +  \PX(t). \\
\end{eqnarray*}
Using $(a+b)^2 \leq 2 a^2 + 2 b^2$, the bias term $T_1$ defined in \eqref{eq:biasf} is upper bounded by
$$
T_1 \leq  \underbrace{2 \left\{\EE \left[\frac{ p-\hat{p}_{n,l}^{(1)}}{1-\hat{p}_{n,l}^{(1)}}\right]\right\}^2 \int_{\RR}  \left| \PQ(\dnl t) (1-\PX(t)) \right|^2 dt}_{:=T_{1,1}} + \underbrace{
2 \int_{\RR}  \left| [\PQ(\dnl t)-1] \PX(t)   \right|^2 dt}_{:=T_{1,2}}, 
$$
We may apply the Cauchy-Schwarz inequality and Lemma \ref{lem:A2} of the appendix, Section \ref{sec:appendix} and obtain:
$$
 \left\{\EE \left[\frac{ p-\hat{p}_{n,l}^{(1)}}{1-\hat{p}_{n,l}^{(1)}}\right]\right\}^2 \leq \EE \left[\frac{ p-\hat{p}_{n,l}^{(1)}}{1-\hat{p}_{n,l}^{(1)}}\right]^2 \lesssim h_{n,l}^{2 s_l+1}.
$$
The function $\lvert\PX\lvert$ is bounded by $1$ so that
$$
T_{1,1} \lesssim h_{n,l}^{2s_l+1} \int_{\RR} \left| \PQ(\dnl t) (1-\PX(t)) \right|^2 dt \lesssim h_{n,l}^{2s_l+1} \int_{\RR} \left| \PQ(\dnl t)  \right|^2 dt \lesssim h_{n,l}^{2s_l+1} \dnl^{-1} \|\PQ\|_{_2}^2.
$$
Moreover, for a given smoothness parameter $s_l$, the estimation of $p$ is ``easier" (faster) than the estimation of $f$, which is translated by $h_{n,l} = o \left( \dnl \right)$. We conclude that:
$$T_{1,1} \lesssim h_{n,l}^{2s_l} \sim n^{-2 s_l/(2s_l+2\nu)} = o(\rho^2_{n,l}).$$
The upper bound of the second term $T_{1,2}$ uses the smoothness assumption on $f \in \Hs$ and the construction of the kernel $Q$: applying Equation \eqref{eq:kernel_f} with $\alpha=s$, we obtain:
$$
T_{1,2} \leq 2 \int_{\RR} | \dnl t|^{2 s} |\PX(t)|^2 dt \lesssim \dnl^{2s} =\mathcal{O} \left( n^{-2s/(2s_l+2\nu+1)} \right).
$$
Since $l \geq j^\star$, we deduce that $s_l \leq s$ and the additional log term in the definition of $\rho_{n,l}$ in \eqref{eq:penalty2} permits to conclude that
$$
T_{1,2} \leq  \dnl^{2s_l} =o(\rho_{n,l}^2).
$$
Therefore, $T_1$ is smaller than $\rho_{n,l}^2/8$ for $n$ large enough.
\\
We can conclude the proof using the triangle inequality
$$
 \| \hPXn{l} - \EE \hPXn{l} \|_{_{2}} \ge 
 \| \hPXn{l} - \PX\|_{_{2}}    - \| \EE \hPXn{l} - \PX\|_{_{2}}.
$$
Hence
$$
\| \hPXn{l} - \PX\|_{_{2}}  \geq \frac{\rho_{n,l}}{2 }  \Longrightarrow  \| \hPXn{l} - \EE \hPXn{l} \|_{_{2}} \geq  \frac{\rho_{n,l}}{2 }  - \sqrt{T_1} \geq \frac{\rho_{n,l}}{2 \sqrt{2}}.
$$
\hfill $\square$
\end{proof}

\subsubsection{Concentration inequality}

From Proposition \ref{prop:inclusion}, we can see that the Lepskii rule will perform well if we succeed in bounding the probability 
$$
\PP \left(   \|\hPXn{l} - \EE \hPXn{l}\|_{_{2}}^2  \ge \frac{\rho_{n,l}^2}{8} \right).
$$
For a given $t \in \RR$, we can rewrite this expression as
{\scriptsize \begin{eqnarray*}
\hPXn{l}(t) - \EE \hPXn{l} (t) & =&
 \PQ(\dnl t) \left[ \frac{\hPZ^{(2)}(t)-\hpln \PU(t)}{(1-\hpln)\PU(t)} \right] - \PQ(\dnl t) \EE \left[ \frac{\hPZ^{(2)}(t)-\hpln \PU(t)}{(1-\hpln)\PU(t)} \right] \\
& =&  \frac{\PQ(\dnl t) }{\PU(t)}  \left[ \frac{\hPZ^{(2)}(t)-\hpln \PU(t)}{1-\hpln} \right]  - \frac{\PQ(\dnl t) }{\PU(t)}  \PZ(t) \EE \left[\frac{1}{1-\hpln}\right] +\PQ(\dnl t)  \EE \left[\frac{\hpln}{1-\hpln}\right] \\
& = & \frac{\PQ(\dnl t) }{\PU(t)} \left[ \frac{\hPZ^{(2)}(t)}{1- \hpln} - \PZ(t) \EE \left[\frac{1}{1-\hpln}\right]\right] - \PQ(\dnl t) 
\left[ \frac{\hpln}{1-\hpln} - \EE \left[\frac{\hpln}{1-\hpln}\right]\right]\\
&  = & T_{2,1}+T_{2,2}-T_{2,3},\end{eqnarray*}}
where 
$$
T_{2,1}:= \frac{\PQ(\dnl t)}{\PU(t)} \frac{\hPZ^{(2)}(t) - \PZ(t)}{1- \hpln}  \quad T_{2,2} :=  \frac{\PQ(\dnl t) \PZ(t)}{\PU(t)}
\left[ \frac{1}{1- \hpln}-  \EE \left[\frac{1}{1-\hpln}\right]\right],
$$
and 
$$
T_{2,3}:= \PQ(\dnl t) 
\left[ \frac{\hpln}{1-\hpln} - \EE \left[\frac{\hpln}{1-\hpln}\right]\right].
$$
The terms $T_{2,2}$ and $T_{2,3}$ can be upper bounded easily while $T_{2,1}$ deserves a specific attention.

\paragraph{Study of $T_{2,2}$:}
We first remark that  $\EE \left[ \frac{1}{1-\hpln}\right]$ is close to $\frac{1}{1-p}$ using the upper bound:
\begin{equation}\label{eq:maj_triv}
\left| \frac{1}{1-\hpln} - \frac{1}{1-p}\right| \leq \frac{|p-\hpln|}{1-p}(\log n)^{a},
\end{equation}
where the inequality derives from $\hpln\leq 1-\tau_n$ and \eqref{eq:taun}. Therefore, applying Theorem \ref{theo:estimp} with the smoothness parameter $s_l$ we deduce that
\begin{equation}\label{eq:esperance}
\left| \EE \left[\frac{1}{1-\hpln}\right]-\frac{1}{1-p} \right| \lesssim (\log n)^{a} n^{-(s_l+1/2)/(2s_l+2\nu)},
\end{equation}
Consider a constant $c>0$, the term $T_{2,2}$ is then handled as follows:
\begin{eqnarray*}
\PP \left( \|T_{2,2}\|_2 \geq c\rho_{n,l} \right) &\leq& \PP \left( \left| \frac{1}{1-\hpln} - \EE \left[ \frac{1}{1-\hpln}\right]\right| \geq  \frac{ c \rho_{n,l}}{\| \PQ(\dnl .) \PZ \PU^{-1}\|_2} \right) \\
& \leq &  \PP \left( \left| \frac{1}{1-\hpln} - \EE \left[ \frac{1}{1-\hpln}\right]\right| \geq  \frac{ c \rho_{n,l}}{\| \PQ(\dnl .) (p+(1-p)\PX )\|_2} \right) \\
& \leq & \PP \left( \left| \frac{1}{1-\hpln} - \EE \left[ \frac{1}{1-\hpln}\right]\right| \geq   \frac{c}{2\|\PQ\|_2} \rho_{n,l} \dnl^{1/2}\right),
\end{eqnarray*}
where we use the upper bound of the $L_2$ norm: $\| \PQ(\dnl .)  (p+(1-p)\PX ) \|_2^2 \leq 4 \int_{\RR} |\PQ(\dnl t)|^2 dt \leq 4\dnl^{-1} \|\PQ\|_2^2$.
We now use Equation \eqref{eq:esperance} and obtain:
{\scriptsize
\bea
\PP \left(\left| \frac{1}{1-\hpln} - \EE \left[ \frac{1}{1-\hpln}\right]\right| \right.&\left. \geq   \frac{c}{2\|\PQ\|_2} \rho_{n,l} \dnl^{1/2}\right) \\
&\leq
\PP \left(\left| \frac{1}{1-\hpln} - \frac{1}{1-p} \right| + \left| \frac{1}{1-p} -\EE \left[ \frac{1}{1-\hpln}\right]\right| \geq   \frac{c}{2\|\PQ\|_2} \rho_{n,l} \dnl^{1/2}\right)
\\
& \leq  \PP \left(\left| \frac{1}{1-\hpln} - \frac{1}{1-p} \right| \geq   \frac{c}{2\|\PQ\|_2} \rho_{n,l} \dnl^{1/2} - C (\log n)^{a} n^{-(s_l+1/2)/(2s_l+2\nu)} \right)\\
& \leq \PP \left(\left| \frac{1}{1-\hpln} - \frac{1}{1-p} \right| \geq   \frac{c}{4\|\PQ\|_2} \rho_{n,l} \dnl^{1/2} \right),
\eea
}
for $n$ large enough because the approximation term involved in the upper bound of Equation \eqref{eq:esperance} is negligible comparing to $\rho_{n,l} \dnl^{1/2}$.  Using again \eqref{eq:maj_triv},
we obtain that for a sufficiently small constant $\eta$ (independent on $n$ and $l$):
$$
\PP \left( \|T_{2,2}\|_2 \geq c\rho_{n,l} \right)  \leq \PP \left( |p-\hpln| \geq \eta \rho_{n,l} \dnl^{1/2}\right).
$$
In order to apply the Bernstein inequality (see Theorem \ref{theo:bernstein} in the appendix Section \ref{sec:appendix}) , let us remark first that the penalty $\kappa_{n,l}$ defined by \eqref{eq:penalty} satisfies 
$$\kappa_{n,l}=o\left(\rho_{n,l} \dnl^{1/2} \right).$$ 
Since for $n$ large enough we have $p\le 1-\tau_n$ , the truncated estimator satisfies
$$\|\hpln-p\|\le \| \hat{p}_{n/2,l}-p\|$$ where $\hpln$ is given by \eqref{eq:pnj1}. Thus we can conclude that
\ben\label{eq:T22}
\PP \left( \|T_{2,2}\|_2 \geq c\rho_{n,l} \right) \lesssim n^{-\beta^2/64}.
\een

\paragraph{Study of $T_{2,3}$:}
This study is of the same nature because  $\EE \left[ \frac{\hpln}{1-\hpln}\right]$ is close to $\frac{p}{1-p}$: Inequality \eqref{eq:esperance} reads:
\begin{equation*}
\left| \EE \left[\frac{\hpln}{1-\hpln}\right]-\frac{p}{1-p} \right| 
= \left| \EE \left[ \frac{p-\hpln}{(1-p)(1-\hpln)} \right]\right| 
 \lesssim (\log n)^{a} n^{-(s_l+1/2)/(2s_l+2\nu)},
\end{equation*}
Hence
\begin{eqnarray*}
\PP \left( \|T_{2,3}\|_2 \geq c\rho_{n,l} \right) &\leq& \PP \left( \left| \frac{\hpln}{1-\hpln} - \EE \left[ \frac{p}{1-\hpln}\right]\right| \geq  \frac{ c \rho_{n,l}}{2\| \PQ(\dnl .)\|_2} \right) \\
& \leq & \PP \left( \left| \frac{\hpln}{1-\hpln} - \EE \left[ \frac{p}{1-\hpln}\right]\right| \geq   \frac{c}{2\|\PQ\|_2} \rho_{n,l} \dnl^{1/2}\right),
\end{eqnarray*}
We  use again Equation \eqref{eq:esperance} and obtain that:
\bea
\PP \left(\left| \frac{\hpln}{1-\hpln} - \EE \left[ \frac{\hpln}{1-\hpln}\right]\right| \geq \right.&\left.   \frac{c}{2\|\PQ\|_2} \rho_{n,l} \dnl^{1/2}\right)\\
 &\leq  
\PP \left(\left| \frac{\hpln}{1-\hpln} - \frac{p}{1-p} \right| + \left| \frac{p}{1-p} -\EE \left[ \frac{\hpln}{1-\hpln}\right]\right| \geq   \frac{c}{2\|\PQ\|_2} \rho_{n,l} \dnl^{1/2}\right)
\\
& \leq  \PP \left(\left| \frac{\hpln}{1-\hpln} - \frac{p}{1-p} \right| \geq   \frac{c}{4 \|\PQ\|_2} \rho_{n,l} \dnl^{1/2} \right),
\eea
and similar arguments used for $T_{2,2}$ yields:
\ben\label{eq:T23}
\PP \left( \|T_{2,3}\|_2 \geq c\rho_{n,l} \right) \lesssim n^{-\beta^2/64}.
\een

\paragraph{Study of $T_{2,1}$:}
We recall that the support of $\PQ$ is $[-1;1]$, which implies that
$$
 \|T_{2,1} \|_{_2}^2 = \int_{-\dnl^{-1}}^{\dnl^{-1}} \frac{|\PQ(\dnl t)|^2}{|\PU(t)|^2} 
  \frac{\left| \hPZ^{(2)}(t) - \PZ(t)\right|^2}{(1- \hpln)^2}  dt.
$$
We now define $M_{n,l}$  as the supremum of the empirical process $ \hPZ^{(2)}(t) - \PZ(t)$ when $t$ varies between $-\dnl^{-1}$ and $\dnl^{-1}$:
$$
M_{n,l} := \sup_{|t| \leq \dnl^{-1}} \left| \hPZ^{(2)}(t) - \PZ(t)\right|,
$$
and write
$$ \|T_{2,1} \|_{_2} \leq M_{n,l} \times \frac{1}{(1-\hpln)} \times \left\| \frac{\PQ(\dnl t)}{\PU(t)}\right\|_{_2}.$$
We then deduce that
\begin{eqnarray*}
\PP \left( \|T_{2,1} \|_{_2} \geq \rho_{n,l}\right) &\le &\PP \left(  \left\| \frac{\PQ(\dnl t)}{\PU(t)}  \right\|_2  \frac{ M_{n,l}}{1- \hpln}\geq \rho_{n,l} \right) \\
& \leq  & \PP \left( \left\| \PQ(\dnl .) \PU(.)^{-1}\right\|_{_2} M_{n,l}  \geq \tau_n  \rho_{n,l} \right).\\
 \end{eqnarray*}
To upper bound the $\LL^2$ norm of $\PQ(\dnl .) \PU^{-1}$, we follow a standard argument of \cite{Fan} using that $g\in\Hnu$ and deduce that
 $$
 \left\| \PQ(\dnl .) \PU(.)^{-1}\right\|_{_2}^2 =\int_{\RR} \frac{|\PQ(\dnl t)|^2}{|\PU(t)|^2} dt \leq \frac{1}{\dnl} \int_{-1}^1 \frac{|\PQ(\xi)|^2}{|\PU(\xi\dnl^{-1})|^2} d\xi \leq \frac{ \|\PQ\|_{_2}^2}{d_2^{-2}} \dnl^{-(2 \nu+1)}.
 $$
This inequality yields
\begin{equation}\label{eq:t_21}
 \PP \left( \|T_{2,1} \|_{_2} \geq \rho_{n,l}\right) \leq \PP \left( M_{n,l}\geq \frac{d_2}{\|\PQ\|_{_2}} \tau_n  \rho_{n,l} \dnl^{1/2+\nu} \right).
 \end{equation}
 It remains to obtain a deviation inequality on $M_{n,l}$, which is the supremum norm of $\hPZ^{(2)}-\PZ$ on the interval $[-\dnl^{-1},\dnl^{-1}]$. For this purpose, we could try to use the roadmap of applying the Talagrand inequality associated to a chaining strategy (see \cite{BLM} for example) but these general theorems will not lead to a satisfactory deviation bound because of our assumption on $U$, which is not a sub-Gaussian random variable.
 
  Nevertheless, it is possible to exploit the feature of the process $(\hPZ^{(2)}(t) - \PZ(t))_{-\dnl^{-1} \leq t \leq \dnl}$, which is an average of  centered bounded random processes. We handle
 a discretization of the interval $[-\dnl^{-1},\dnl]$ and define $t \longmapsto W_n(t)$ by:
 $$
 \forall t \in \RR: \qquad 
 W_n(t) = \hPZ^{(2)}(t) - \PZ(t).
 $$ We introduce a  parameter $\alpha>0$ that will be chosen small enough below. The control of  $M_{n,l}$ relies on the simple remark that:
 $$
 W_n(t+u)-W_n(t) = \frac{1}{n-n'+1} \sum_{j=n'}^n \left[ e^{\ii t Z_j} - e^{\ii (t+u) Z_j}\right] - \mathbb{E} \left[ e^{\ii t Z}  - e^{\ii (t+u) Z } \right],
 $$ where we set $n'=[n/2]+1$.
 The elementary remark $|e^{\ii a } - e^{\ii b}| \leq |a-b|$ for all $(a,b) \in \RR^2$ yields
\ben
\label{eq:maj_ecart}
| W_n(t+u)-W_n(t)| \leq |u| \left( \frac{1}{n-n'+1} \sum_{j=n'}^n |Z_j| +  \mathbb{E}  |Z|\right).
\een
 In particular, if we define the event
 $$
 \Omega_n := \left\{   \frac{1}{n-n'+1} \sum_{j=n'}^n |Z_j| < 2 \mathbb{E} |Z| \right\}.
 $$
The Tchebychev inequality  applied to the $\LL^2$-random variables $(Z_j)_{n'\le j\le n}$ shows that:
 $$
\PP(\Omega_n^{c}) \leq \frac{\mathbb{V}ar(|Z|)}{n \mathbb{E}[|Z|]^2} \leq \frac{c}{n},
 $$
 with $c=\mathbb{V}ar(|Z|) \left\{ \mathbb{E}[|Z|]\right\}^{-2}$. Obviously, $\mathbb{E}[|Z|] > 0$ because $Z=U+AX$ and $Z$ cannot be a.s. $0$ because $U$ is assumed to satisfy $\Hnu$.
 The important point on $\Omega_n$ is that Inequality \eqref{eq:maj_ecart} does not depend on $t$, so that $  \forall s >0$:
 \bean
 \PP\Bigl(\sup_{ |t|\le \dnl^{-1} \, ,|u| \leq \alpha } |W_n(t+u)-W_n(t)| > s \Bigr)& \le  \PP\Bigl(\alpha \Bigl( \frac{1}{n-n'} \sum_{j=n'}^n |Z_j| +  \mathbb{E}  |Z|\Bigr) > s \cap \Omega_n \Bigr) 
+ \PP\left( \Omega_n^{c} \right) \\
 & \leq  \mathbf{1}_{3 \alpha \mathbb{E}[|Z|]  > s} + \frac{c}{n}
\label{eq:discretisation}
 \eean
 We now define the threshold  $s_{n,l}$ and the window size $\alpha_{n,l}$ by:
\ben\label{def:seuil}
s_{n,l} := \frac{1}{2} \frac{d_2}{\|\PQ\|_{_2}} \tau_n  \rho_{n,l} \dnl^{1/2+\nu}  \qquad \text{and} \qquad 
\alpha_{n,l} :=  \frac{s_{n,l}}{3 \mathbb{E}[|Z|]}.
\een
 We are naturally driven to handle a discrete  grid $\mathcal{T}_{n,l}$ regularly spaced from $-\dnl^{-1}$ to $\dnl^{-1}$ whose step-size is $\alpha_{n,l}$. We have
 $$
 |\mathcal{T}_{n,l}| \leq \frac{2 \dnl^{-1}}{\alpha_{n,l}}= \frac{6 \mathbb{E}[|Z|] \dnl^{-1}}{s_{n,l}} \qquad \text{and} \qquad \forall t \in [-\dnl^{-1},\dnl^{-1}], \quad \exists \tau_t \in \mathcal{T}_{n,l} \, : \qquad |t-\tau_t| \leq \alpha_{n,l}.
 $$
These settings permit to deduce from \eqref{eq:discretisation}  that:
 \begin{equation}\label{eq:c/n2}
 \PP \left( \sup_{|t|\le \dnl^{-1} \, , |u| \leq \alpha_{n,l}} |W_n(t+u)-W_n(t)| > s_{n,l} \right) \leq \frac{c}{n}.
 \end{equation}
Now, for each element of the grid $\tau \in \mathcal{T}_{n,l}$, $W_n(\tau)$ is a mean of $n$ random variables that whose modulus are bounded by $1$. The Hoeffding inequality implies that $\forall \tau \in \mathcal{T}_{n,l}$:
\begin{equation}
\label{eq:appli_hoeffding}
\begin{aligned}
 \PP \bigl( |W_n(\tau)|>s_{n,l} \bigr) 
&\le \PP\left( \Bigl| \frac{1}{n-n'}\sum_{j=n'}^n \cos(\tau Z_j) -\EE(\cos(\tau Z)) \Bigr| <\frac{s_{n,l}}{2}\right)\\
&\quad\quad+\PP\left(\Bigl| \frac{1}{n-n'}\sum_{j=n'}^n \sin(\tau Z_j) -\EE(\sin(\tau Z)) \Bigr|  <\frac{s_{n,l}}{2}\right)\\
&\leq 4 \exp\left( - \frac{(n-n') s_{n,l}^2}{8} \right).\end{aligned}
\end{equation}
We now produce an upper bound of $\PP \left( \|T_{2,1}\|_{_2} \geq \rho_{n,l} \right)$. Equation \eqref{eq:t_21} yields:
\begin{eqnarray*}
\PP \left( \|T_{2,1}\|_{_2} \geq \rho_{n,l} \right) 
&=& \PP \left( \sup_{|t|\le \dnl^{-1}} |W_n(t)| \geq 2 s_{n,l} \right) \\
& \leq & \PP \left( \sup_{ |t|\le \dnl^{-1} , |u| \leq h_{n,l}} |W_n(t+u)-W_n(t)| \geq s_{n,l} \quad \text{and} \sup_{|t|\le \dnl^{-1} } |W_n(t)| \geq 2 s_{n,l} \right) \\
& & + \PP \left( \sup_{|t|\le \dnl^{-1} , |u| \leq h_{n,l}} |W_n(t+u)-W_n(t)| \leq s_{n,l} \quad \text{and} \sup_{|t|\le \dnl^{-1} } |W_n(t)| \geq 2 s_{n,l} \right) \\ 
& \leq & \PP \left( \sup_{|t|\le \dnl^{-1} , |u| \leq h_{n,l}} |W_n(t+u)-W_n(t)| \geq s_{n,l} \right) + \PP \left( \sup_{\tau \in \mathcal{T}_{n,l}} |W_n(\tau)|  \geq s_{n,l} \right) \\
 & \leq &\frac{c}{n} + \sum_{\tau \in \mathcal{T}_{n,l}} \PP \left( |W_n(\tau)| \geq s_{n,l} \right) \\
 & \leq &  \frac{c}{n} + \frac{2 \dnl^{-1}}{\alpha_{n,l}} \times 4 \exp\left( - \frac{(n-n') s_{n,l}^2}{8} \right).
\end{eqnarray*}
where we successively applied the triangle inequality, a union bound over $\mathcal{T}_{n,l}$, and then inequalities \eqref{eq:c/n2} and \eqref{eq:appli_hoeffding}. An immediate computation from \eqref{def:seuil} \eqref{eq:taun},  \eqref{eq:penalty2} and \eqref{eq:def_deltanl} shows that:
\bea
n s_{n,l}^2 &= n \frac{d_2^2}{4 \|\PQ\|_{_2}^2}\tau_n^2 \rho_{n,l}^2 \dnl^{1+2 \nu} \\
&=n \frac{d_2^2}{4 \|\PQ\|_{_2}^2} \times (\log n)^{-2a} \times  \beta^2 (\log n)^{2a+1} n^{- 2 s_l/(2 s_l+2 \nu +1)} \times n^{-(1+2 \nu)/(2 s_l+2\nu+1)} \\
&= \frac{\beta^2 d_2^2}{4 \|\PQ\|_{_2}^2}\log(n),
\eea
and
\bea
 \dnl^{-1}\alpha_{n,l}^{-1}
&=\frac{6\E(|Z|)\|\PQ\|_2}{d_2}\beta \log(n)^{-1/2} n^{\frac{3/2+\nu+s_l}{2\nu+2s_l+1}}.
\eea
We therefore deduce that 
\ben\label{eq:T21}
\PP \left( \|T_{2,1}\|_{_2} \geq \rho_{n,l} \right)  \le  \frac{c}{n} +\beta \log(n)^{-1/2} \frac{48\E(|Z|)\|\PQ\|_2}{d_2}n^{-\frac{\beta^2 d_2^2}{24 \|\PQ\|_{_2}^2}+\frac{3/2+\nu+s_l}{2\nu+2s_l+1}}.
\een

\paragraph{Conclusion }

To conclude with a concentration estimate of $\hPXn{l}$ it remains to combine \eqref{eq:T22}, \eqref{eq:T23} and \eqref{eq:T21} to obtain the following result.
\begin{pro}
\label{prop:concentration_f}
Let $f \in \Hs$, $g$ a known density in $\Hnu$ with $\nu>1$, and $l \ge j^\star$, then
\bea
\label{eq:concentration_f}
\PP \left(   \|\hPXn{l} - \EE \hPXn{l}\|_{_{2}}^2  \ge \frac{\rho_{n,l}^2}{8} \right)
&\lesssim n^{-\frac{\beta^2}{64}} +n^{-1}+ \log(n)^{-1/2} n^{-\frac{\beta^2 d_2^2}{24 \|\PQ\|_{_2}^2}+\frac{3/2+\nu+s_l}{2\nu+2s_l+1}}.
\eea
\end{pro}

\subsubsection{Proof of Theorem \ref{theo:lepski_f}}

Let us recall that Proposition \ref{prop:lepski2} and \ref{prop:inclusion} yield:
\bea
\E \|\hfjn - f\|_2
& \lesssim (\log n)^{a+1/2} n^{-s/(2s+2\nu+1)} + \sqrt{ \sum_{m  \ge j^{\star}} \Pro\Bigl(\|\hat{f}_{n,m} - f\|_{_2}   \ge \frac{\rho_{n,m}}{2} \Bigr)}\\
 &\lesssim (\log n)^{a+1/2}n^{-s/(2s+2\nu+1)} + \sqrt{ \sum_{m  \ge j^{\star}} \Pro\Bigl(\|\hPXn{l} - \EE \hPXn{l}\|_{_{2}}^2  \ge \frac{\rho_{n,l}^2}{8}\Bigr) }
\eea
From the definition of the grid $\mathcal{S}_n$  (see the beginning of Section \ref{subsec:lepskif}), it contains a number of order $s_0\varepsilon^{-1} \log(n)$ points and for all $l$, $n^{-\frac{3/2+s_l+\nu}{2s_l+2\nu+1}}\le n^{-1/2}$. Then, we conclude from Proposition \ref{prop:concentration_f} that:
\bea
\E \|\hfjn - f\|_2
 &\lesssim \beta (\log n)^{a+1/2} n^{-s/(2s+2\nu+1)} + \sqrt{s_0 \varepsilon^{-1} \log n\left(n^{-\frac{\beta^2}{64}} +n^{-1}+(\log n)^{-1/2} n^{-\frac{\beta^2 d_2^2}{24 \|\PQ\|_{_2}^2}+\frac{3/2+\nu+s_l}{2\nu+2s_l+1}}\right) }
\eea
Therefore, for $\beta$ large enough:
\bea
\E \|\hfjn - f\|_2
 &\lesssim (\log n)^{a+1/2} n^{-s/(2s+2\nu+1)} ,
\eea
which ends the proof of the almost optimal (up to some log terms) adaptivity of our Lepski's rule. \hfill $\square$

\section{Numerical experiments}
\label{sec:numerics}
\subsection{Simulated data}
Numerical simulation where implemented in Python. For the simulated random variables $U,A,X$ such that $Z=U+AX$, we have considered both Gamma random variables and we recall below their property. 
A  Gamma distribution function with parameters $k$ and $\theta$ (denoted $Gamma(k,\theta)$) admits a density $f(x)$ ($x\in\R$) and characteristic function $\phi(t)$ ($t\in \R$) given by:
$$f(x)=\frac{x^{k-1}e^{-x/\theta}}{\Gamma(k)\theta^k},\quad \Phi(t)=(1-\theta i t)^{-k}.$$
We considered two settings :
\paragraph*{First data set}
$$U\sim Gamma(2,1),\quad A\sim \mathcal{B}er(0.4),\quad X\sim Gamma(\gamma,1),$$
for $\gamma\in\{4,6,8\}$ (see Figure \ref{fig:DensityData1}).
Therefore the smoothness parameter $s$ satisfies $s=\gamma$, this parameter varies while $\nu$ is kept fixed $\nu=2$.

\begin{figure}[h!]\centering
\includegraphics[scale=0.35]{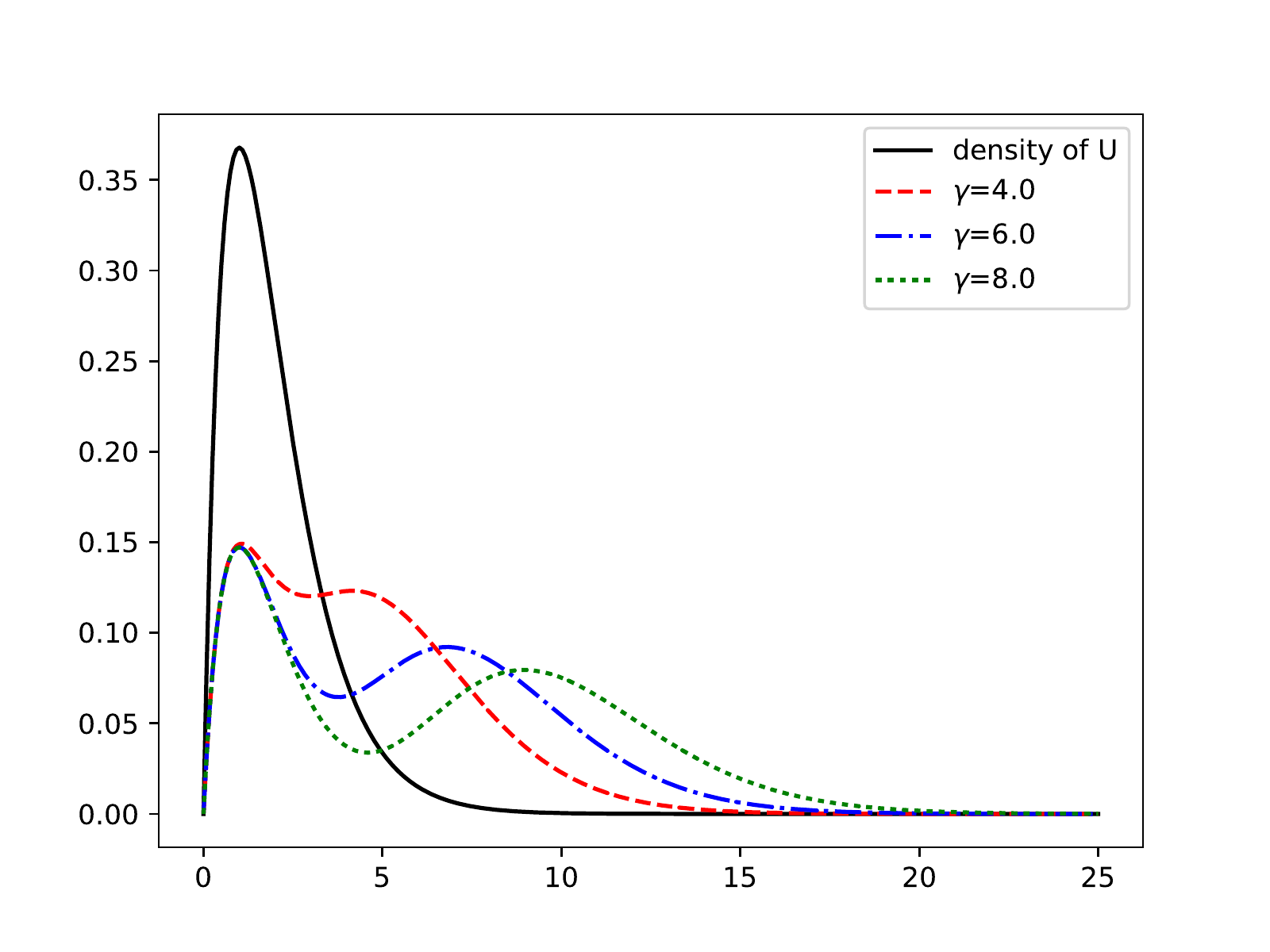}
\caption{Density of the baseline perturbation $U$ (in plain black) and of the observed signal $Z$ for different values of $\gamma$ (dotted curves).}
\label{fig:DensityData1}
\end{figure}
\paragraph*{Second data set}
$$U\sim Gamma(\gamma,1),\quad A\sim \mathcal{B}er(0.4),\quad X\sim Gamma(2,1),$$
for $\gamma\in\{2,4,6,8\}$
Therefore the smoothness parameter $\nu$ satisfies $\nu=\gamma$, this parameter varies while $s$ is kept fixed  $s=2$.

\paragraph{Results - Estimation of $p$} We first focus on the estimation of $p$. For each data sets we performed Monte Carlo replications (with 100 repetitions) in order to approach $\E \Bigl(\|\hat{p}_{n,\hjp} - p \|_2^2\Bigr) $  for different values of the sample size $n\in\{100,200,500,700,1000,2000\}$. Let us remark that the Lepskii procedure is costly in terms of numerical simulations since the criterion for selection requires to compute all the estimators on a grid. An alternative procedure has thus been proposed by \cite{Katkovnik} based on the construction of confidence intervals. However, in the setting of the estimation of a density, this construction is no longer possible. Therefore, our simulations are based on the exact Lepskii decision rule given in \eqref{def:j_chap_p} for $p$ and \eqref{eq:defln} for $f$. Figure \ref{fig:estim_p} gives in a $\log-\log$ scale the estimators of $\E \Bigl(\|\hat{p}_{n,\hjp} - p \|_2^2\Bigr) $ obtained by the Monte Carlo replications as a function of the numbers of observations $n$. The left and right sides of Figure \ref{fig:estim_p} correspond respectively to the first and second data sets.

\begin{figure}[h!]
\centering
\includegraphics[width=0.35\linewidth]{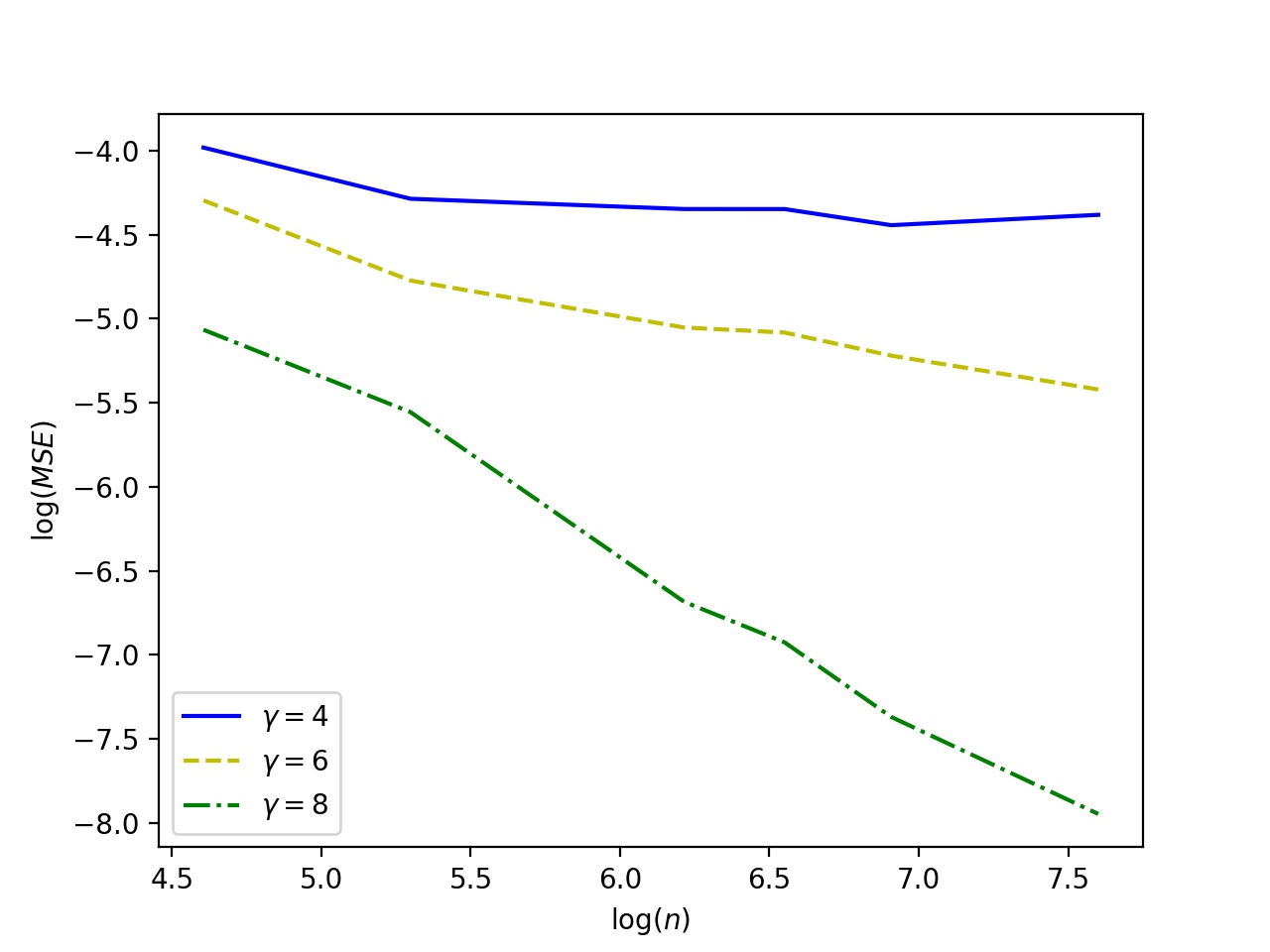}
\includegraphics[width=0.35\linewidth]{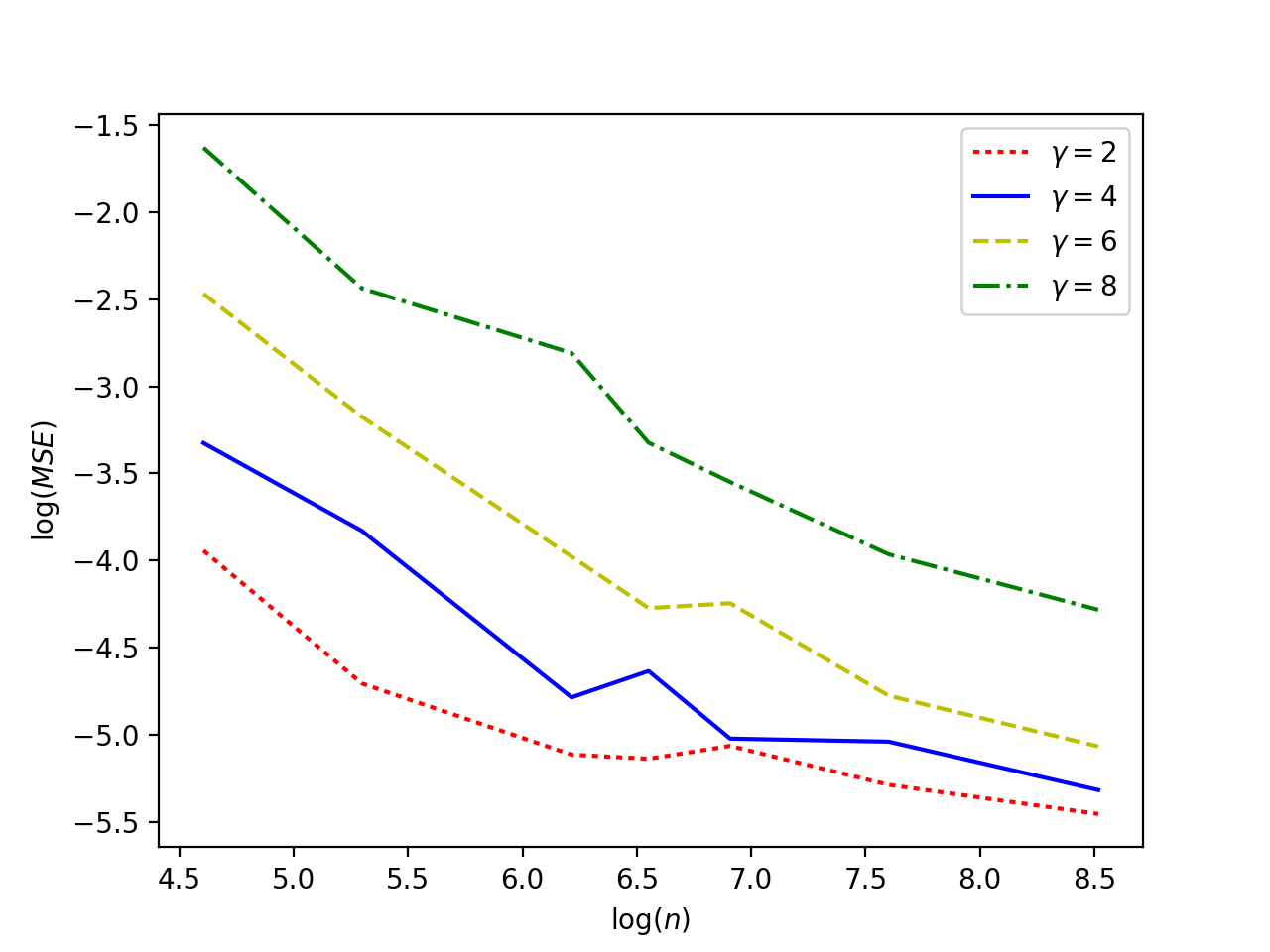}
\caption{Log-Log plot of $n \longmapsto \E \Bigl(\|\hat{p}_{n,\hjp} - p \|_2^2\Bigr) $ obtained by Monte Carlo replications. Left: first data set for several values of $s=\gamma$. Right: second data set for several values of $\nu=\gamma$. For these simulations, we choose $\varepsilon=0.5$ and $\beta=9.0$.}
\label{fig:estim_p}
\end{figure}
In both cases, logarithm of the mean squared error decreases linearly with $log(n)$, which is consistent with Theorem \ref{theo:lepski_p} and \ref{theo:estimp}.
We observe in Figure \ref{fig:estim_p} 
that the estimation of $p$ is easier for large values of $s$ (large $\gamma$) on the left graph, while the estimation gets more difficult as $\nu$ increases (large values of $\gamma$) on the right graph. This phenomenon is completely consistent with the rates derived in Theorem \ref{theo:lepski_p}.

\paragraph{Results - Estimation of $f$ and $\E\Bigl( \|\hat{f}_{n,\hjf} - f \|_2 \Bigr)$} For the non parametric estimation of $f$, we compared the estimation obtained with the Lepskii procedure proposed in this article with a similar one assuming that the mixing parameter $p$ is known. 
In Figure \ref{fig:phi_X_n}, we draw the expected Fourier transform $\phi_X$ and estimators obtained with the Lepskii procedure for different sample sizes $n\in\{100,1000,10000\}$. The left side corresponds to the complete Lepskii procedure \eqref{eq:defln} while in the right side the parameter $p$ is given. In this case, the simulated data correspond to the first data set with $\gamma=4$. On these graphs, we can first see that when the sample size increases, the bandwidth selected in the algorithm is improved since the support of $\hat{\Phi}_{X_{n}}$ widen.

\begin{figure}[h!]
\centering
\includegraphics[width=0.35\linewidth]{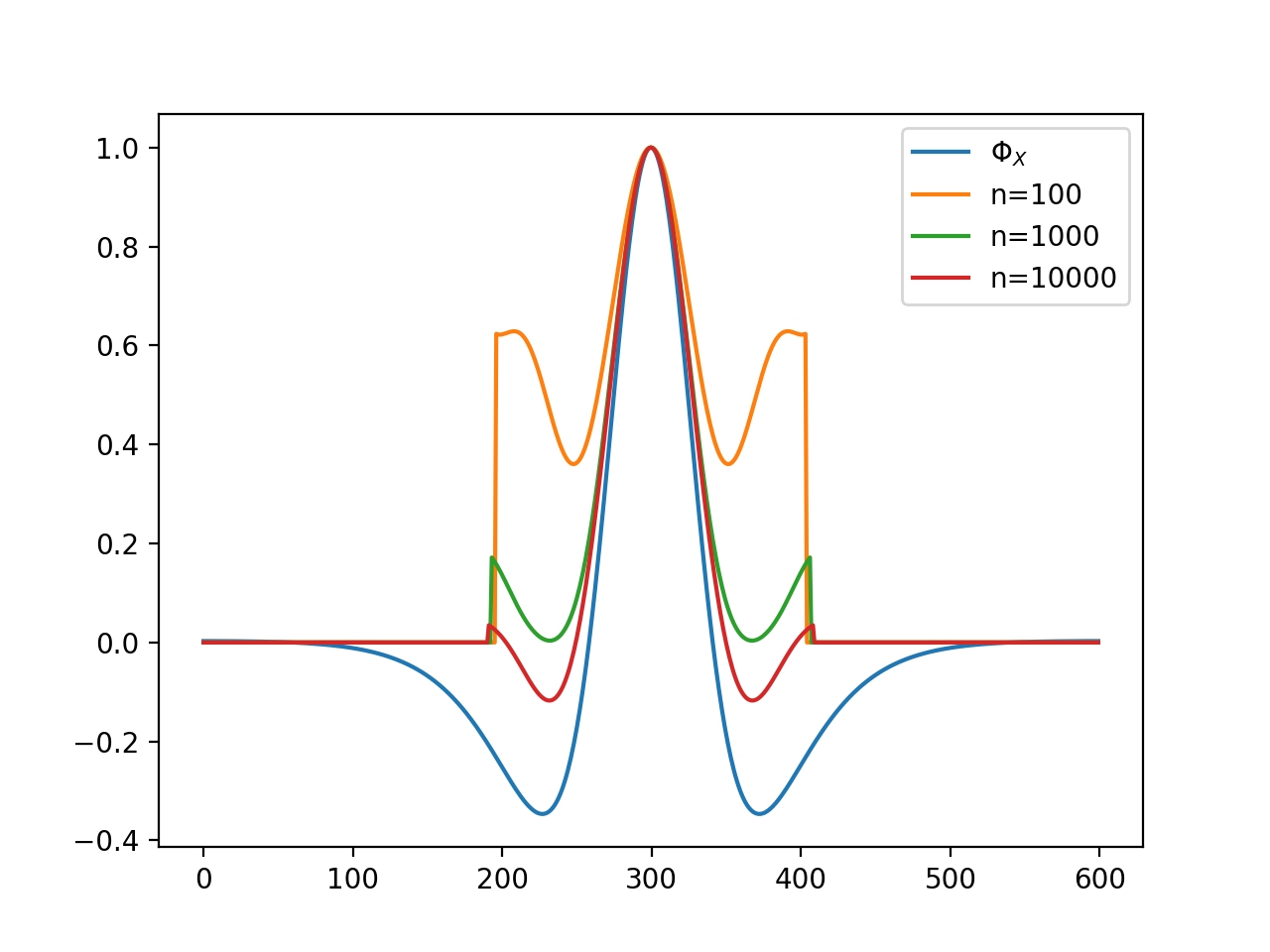}
\includegraphics[width=0.35\linewidth]{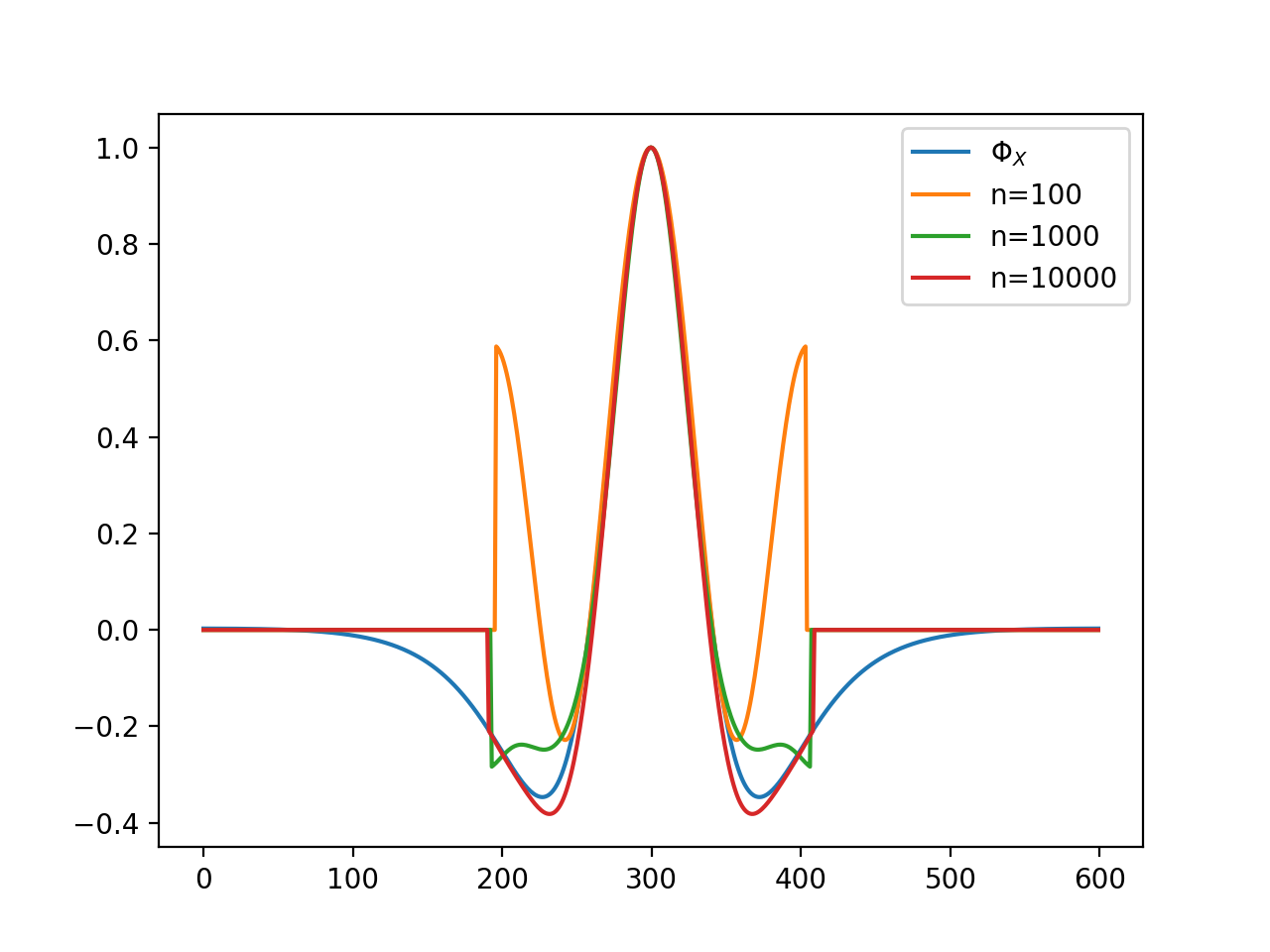}
\caption{True characteristic function $\Phi_X$ as well as its estimations counterpart for different sample sizes $n\in\{100,1000,10000 \}$. Left: Lepskii procedure described in Section \ref{subsec:lepskif}. Right: the true value of the mixing parameter $p$ is assumed to be known. We used datas according to the first data set with $\gamma=4$ and we set $\varepsilon=0.5$ and $\beta=9$.}
\label{fig:phi_X_n}
\end{figure}

We shall observe first that the algorithm leads to better results when the mixing parameter is known, which is not surprising since the estimation of $p$ provides an additional noise in our statistical estimator.
As for the estimation of $p$, we computed with Monte Carlo replications the mean squared estimation error for different sample sizes (see Figure \ref{fig:MSE_f}). Once again we compared our algorithm (left side) with the case where $p$ is known (right side). We find that the estimators are more accurate when $p$ is known and that the regularity $s=\gamma$ of the function $f$ improves the convergence speed.

\begin{figure}[!h]
\centering
\includegraphics[width=0.35\linewidth]{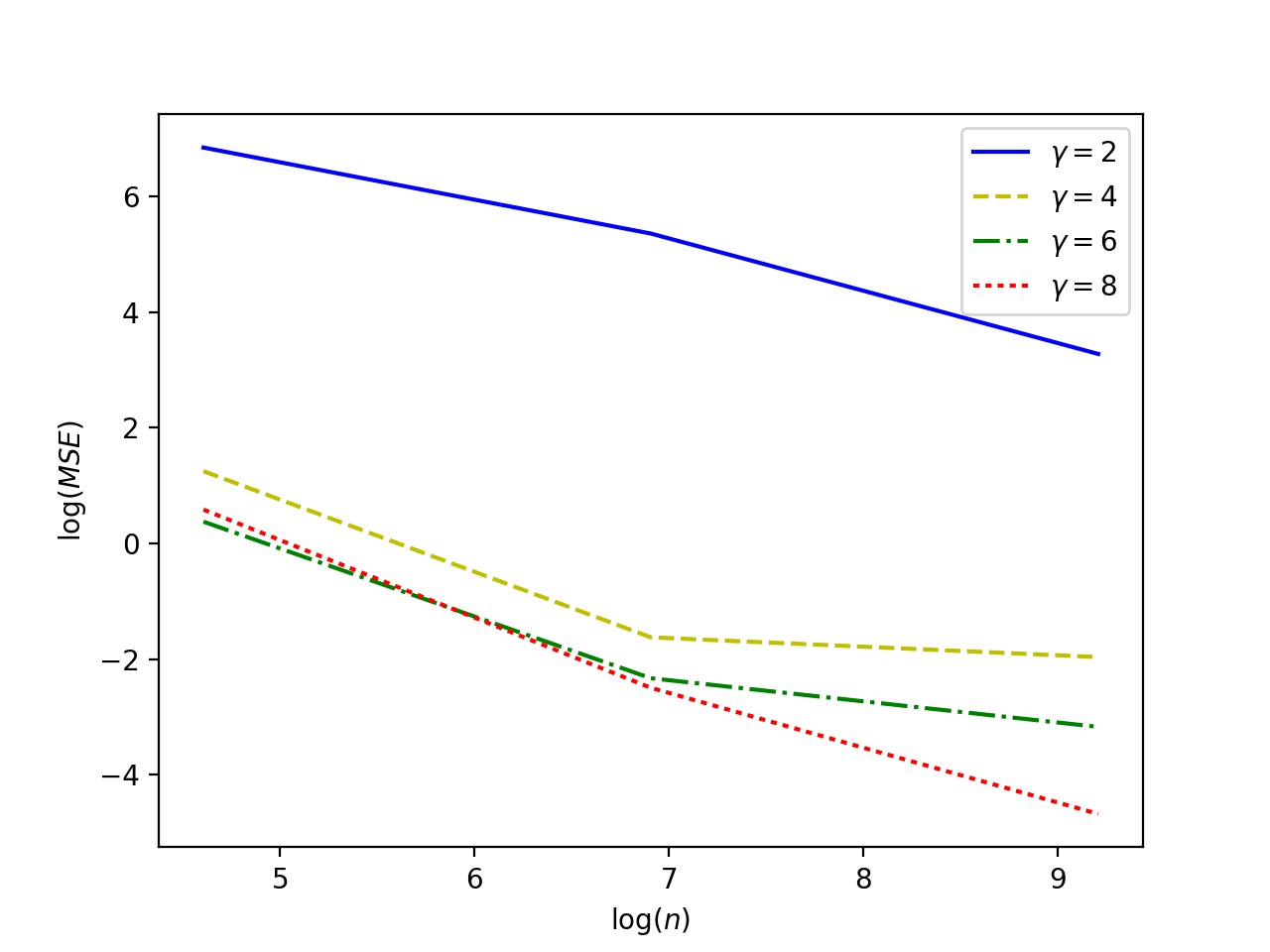}
\includegraphics[width=0.35\linewidth]{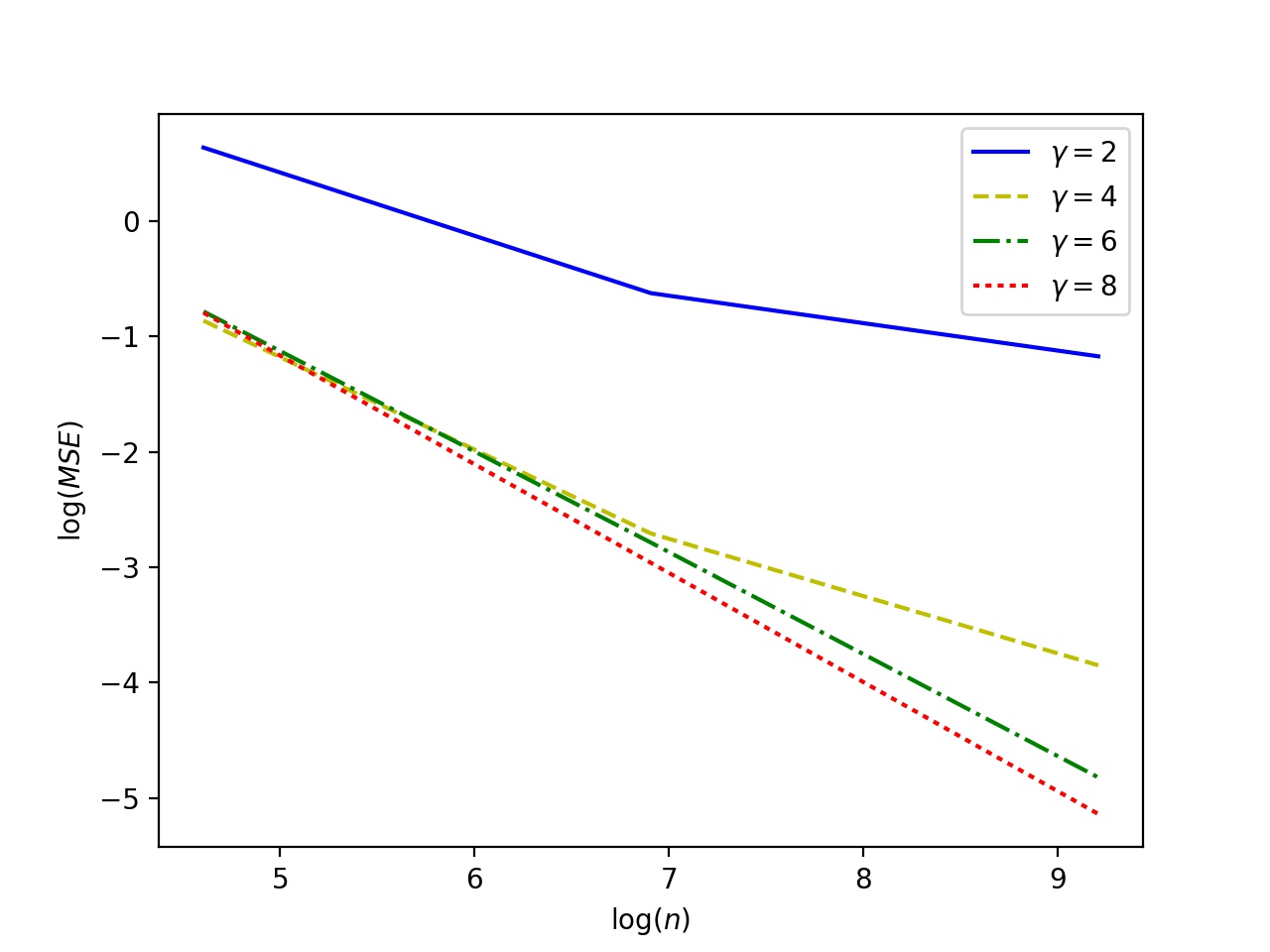}
\caption{On these graphs, we draw in a $\log-\log$ scale the estimators of $\E \|\hat{f}_{n,\hjf} - f \|_2$ obtained by Monte Carlo simulations (for $n\in\{100,1000,10000 \}$). The left panel corresponds to the case where $p$ is unknown while on the right panel, the value of $p$ was given. For these simulations, we used the first data set with different values of $\gamma$ and chose $\beta=9.0$ and $\varepsilon=0.5$.}
\label{fig:MSE_f}
\end{figure}

\subsection{Real data}
\label{subsec:realdata}

\paragraph{Description of the dataset}
The dataset used in this study is coming from fluorescence distribution measured using a flow cytometer instrument (BD Biosciences, LSRII FORTESSA) on cells obtained from human blood.
Lymphocytes are  cells of the immune system useful for cancer cell destruction in immunotherapy (among other), they were extracted from the fresh blood of an healthy individual by the biologist and this cell suspension was then used for the experiment. Technically, the cell suspension was split into two parts. One half was left untouched and the baseline photo emission of untreated cells was recorded by the cytometer. The second half was mixed with a fluorescently labelled antibody (reagent) that specifically binds to the CD27 protein. Then, the photo emission by treated cells was again recorded by the cytometer.
 The amount of reagent binding is reflected by the fluorescence emitted by the cell that is coming only from the antibody treatment (see Figure \ref{fig:cyto_exemple}).
The biological experiment's goal was to assess the expression of the molecule CD27 by human lymphocytes. This protein is expressed at the surface of lymphocytes and reflect their activation status. It is therefore used to estimate the functionality of a lymphocyte population. Methods that would help to decipher the mixture of functionalities among a lymphocyte population is thus of major biological interest. 

In this setting, we use the estimation procedure developed in our paper to infer the percentage of lymphocytes expressing CD27 on their surface and the conditional probability for a cell to express CD27 conditionally to its fluorescence intensity.
\paragraph{Preliminary estimation of the distribution of the noise $U$}
From the recorded fluorescence intensity of untreated cells, we infer the density $g$ of $U$ with a preliminary kernel density estimation using a Gaussian kernel through the \textit{scipy.stats.gaussian$\_$kde} Python software. We then use the recorded fluorescence intensity of treated cells as data for the analysis (see Figure \ref{fig:test3}(a)). 
In this case, the size of the recorded data is $n=12645$.

Of course, our work only deals with the situation where the density $g$ is known, which is unfortunately not possible in our biological situation. Therefore, we admit as a reasonable approximation the estimation of $g$ provided by a preliminary kernel density estimation. We have not treated in our theoretical study the consequences of such a preliminary non-parametric estimation, and we leave this subject open as a future subject of investigation. Such a work would then fall into the field of statistical inverse problems with noise in the operator (see \textit{e.g} \cite{cavalier3} and the references therein).

\paragraph{Estimation of $p$}
For the estimation of the proportion of cells expressing CD27 on their cells' surface, which corresponds to $1-p$, we use the Lepskii procedure proposed in Section \ref{subsec:lepskip}. To verify the convergence of our algorithm, we use a subsampling strategy of our data set and repeat it on permutated versions of the data. 

We observe on Figure \ref{fig:test3}(b) the good behaviour of our algorithm: it produces a sharp estimation even using the half of the data and  it converges to the value $\hat{p}_{n,\hat{j}_n^p}=0.31$. We aim at comparing this value with the measurement classically used in cytometric analysis. Usually biologist use a quantity called \textit{percentage of positive cells} with reflects the percentage of marked cells whose fluorescent intensity is higher that a given threshold. This threshold is computed as the 95th percentile of the density $g$ (see Figure \ref{fig:test3}(a)). On this data, the percentage of positive cells is $0.85$, which is much larger than the quantity we estimate $1-\hat{p}_{n,\hat{j}_n^p}=0.69$.

\begin{figure}[h!]
\hfill
\begin{minipage}{0.3\linewidth}\centering
\includegraphics[scale=0.32]{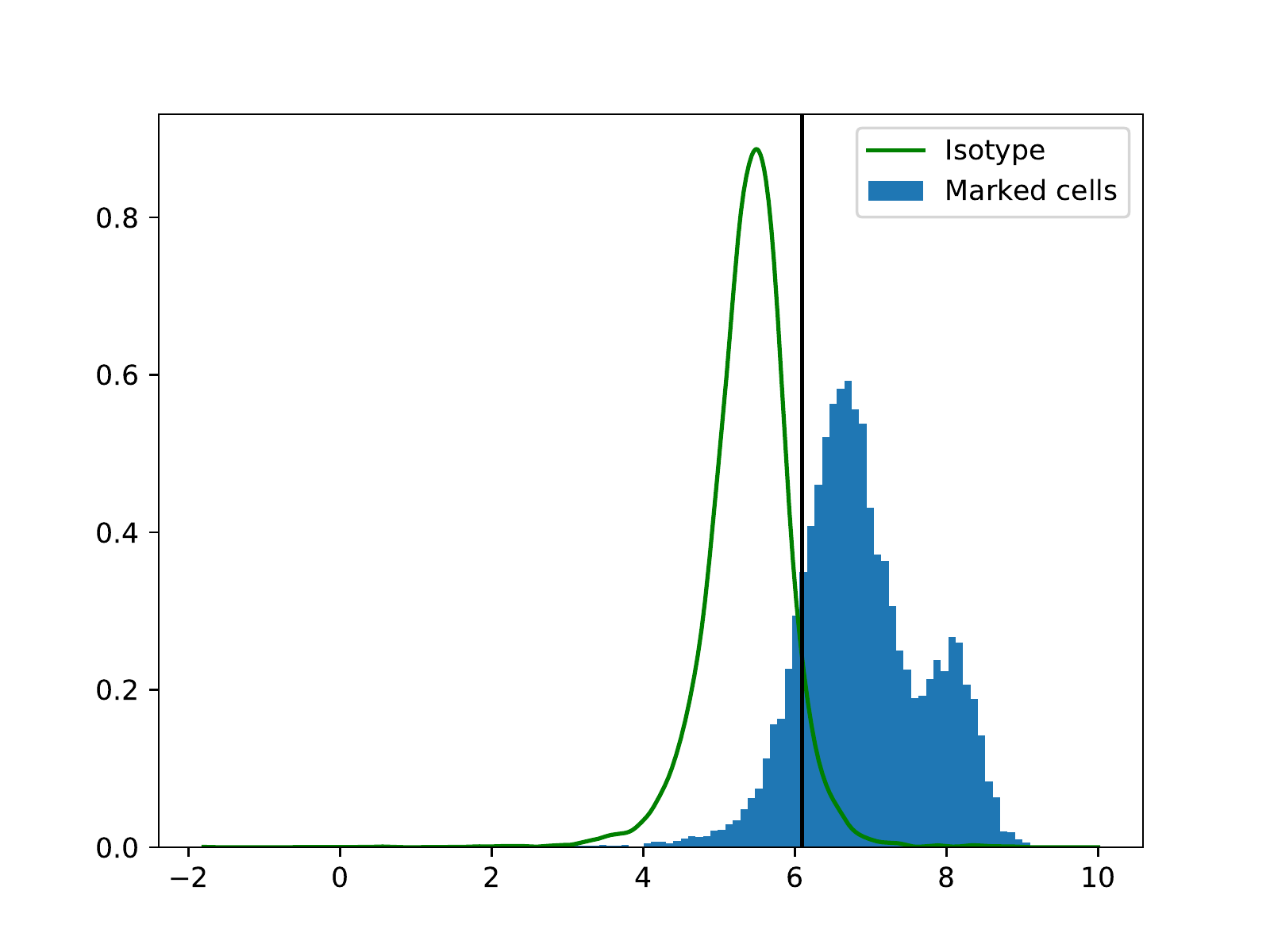}\\
(a)
\end{minipage}
\hfill
\begin{minipage}{0.3\linewidth}\centering
\includegraphics[scale=0.32]{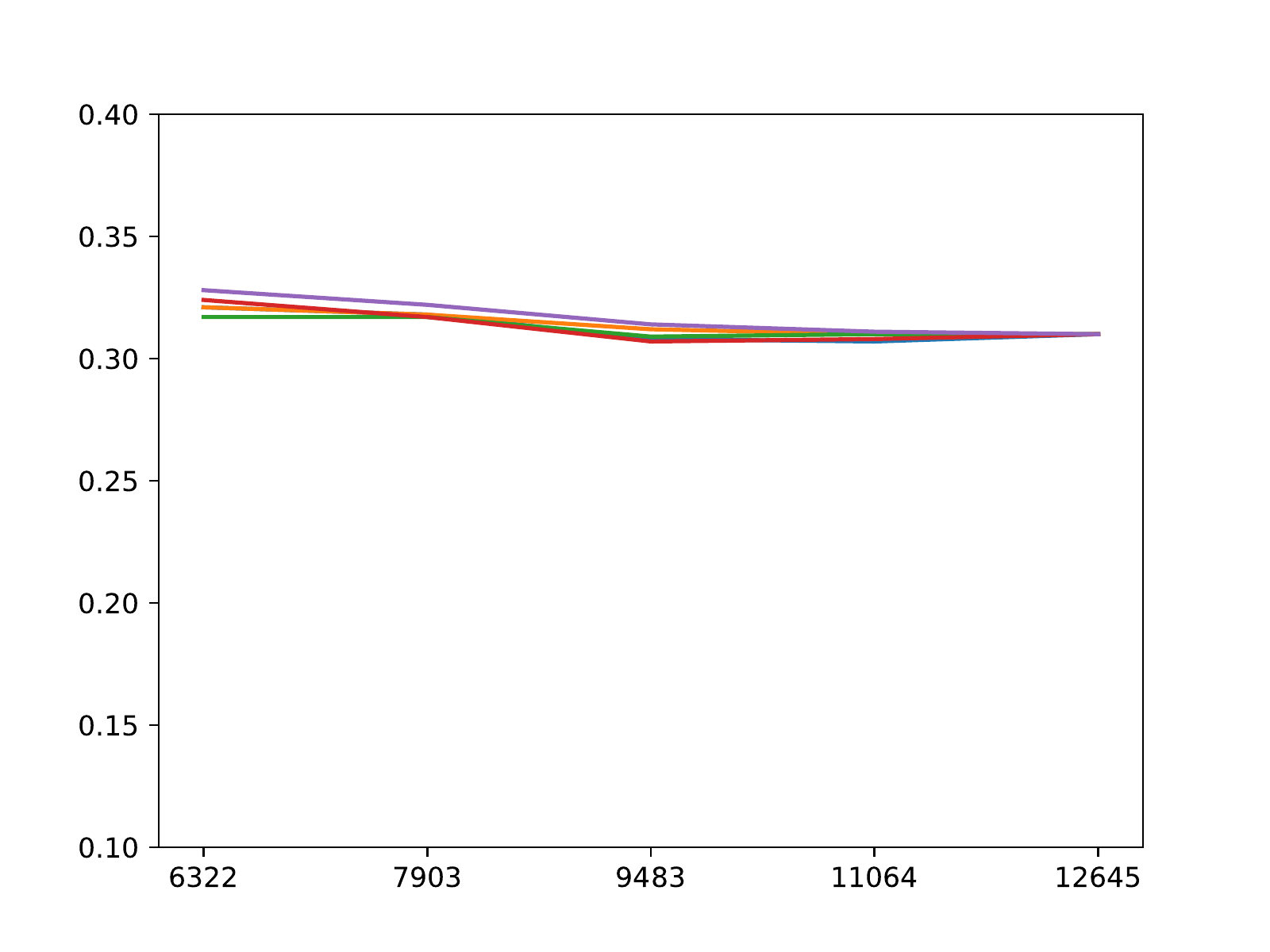}\\
(b)
\end{minipage}
\hfill
\begin{minipage}{0.3\linewidth}\centering
\includegraphics[scale=0.32]{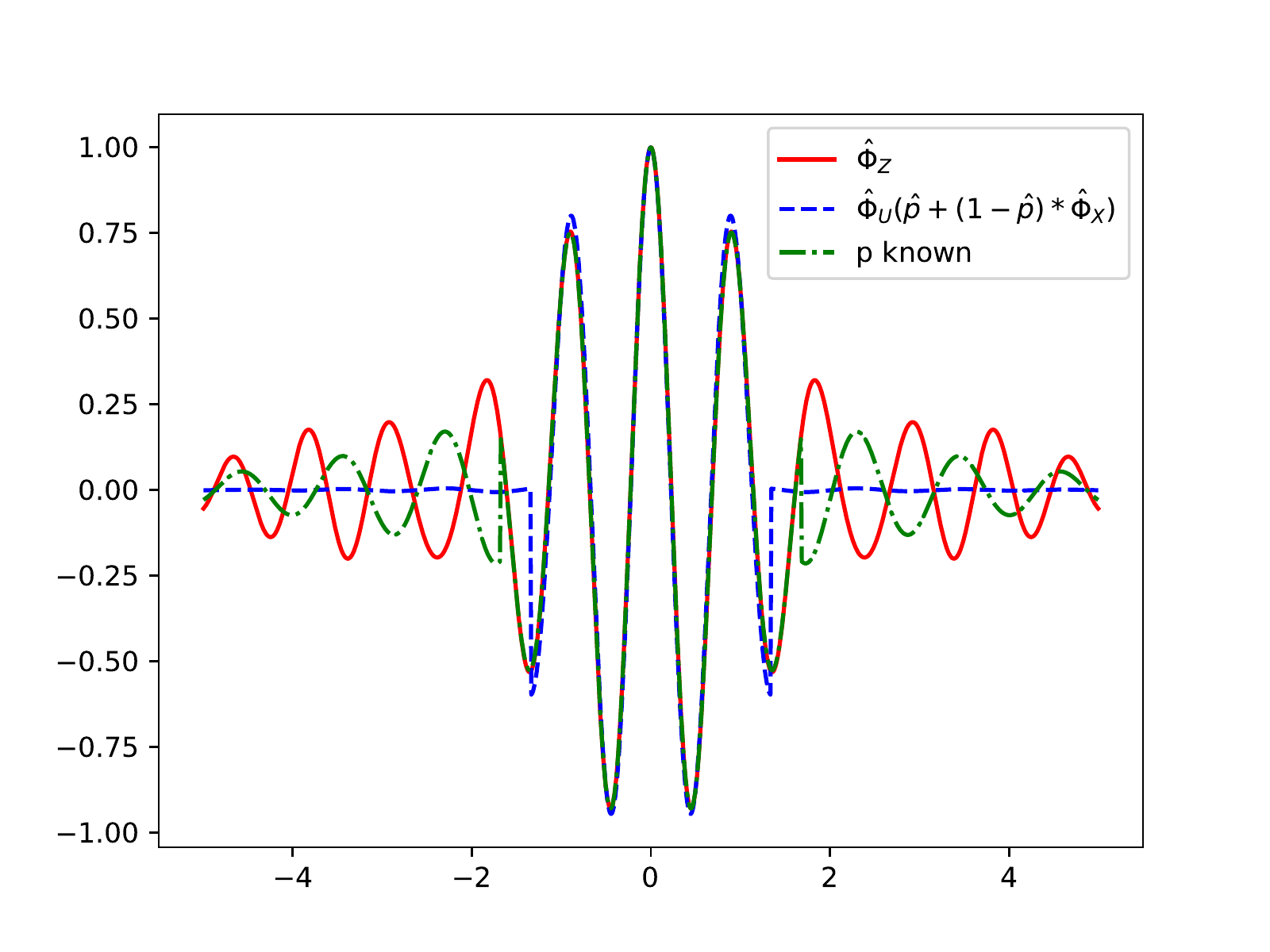}\\
(c)
\end{minipage}
\hfill
\\
\hfill
\begin{minipage}{0.3\linewidth}\centering
\includegraphics[scale=0.32]{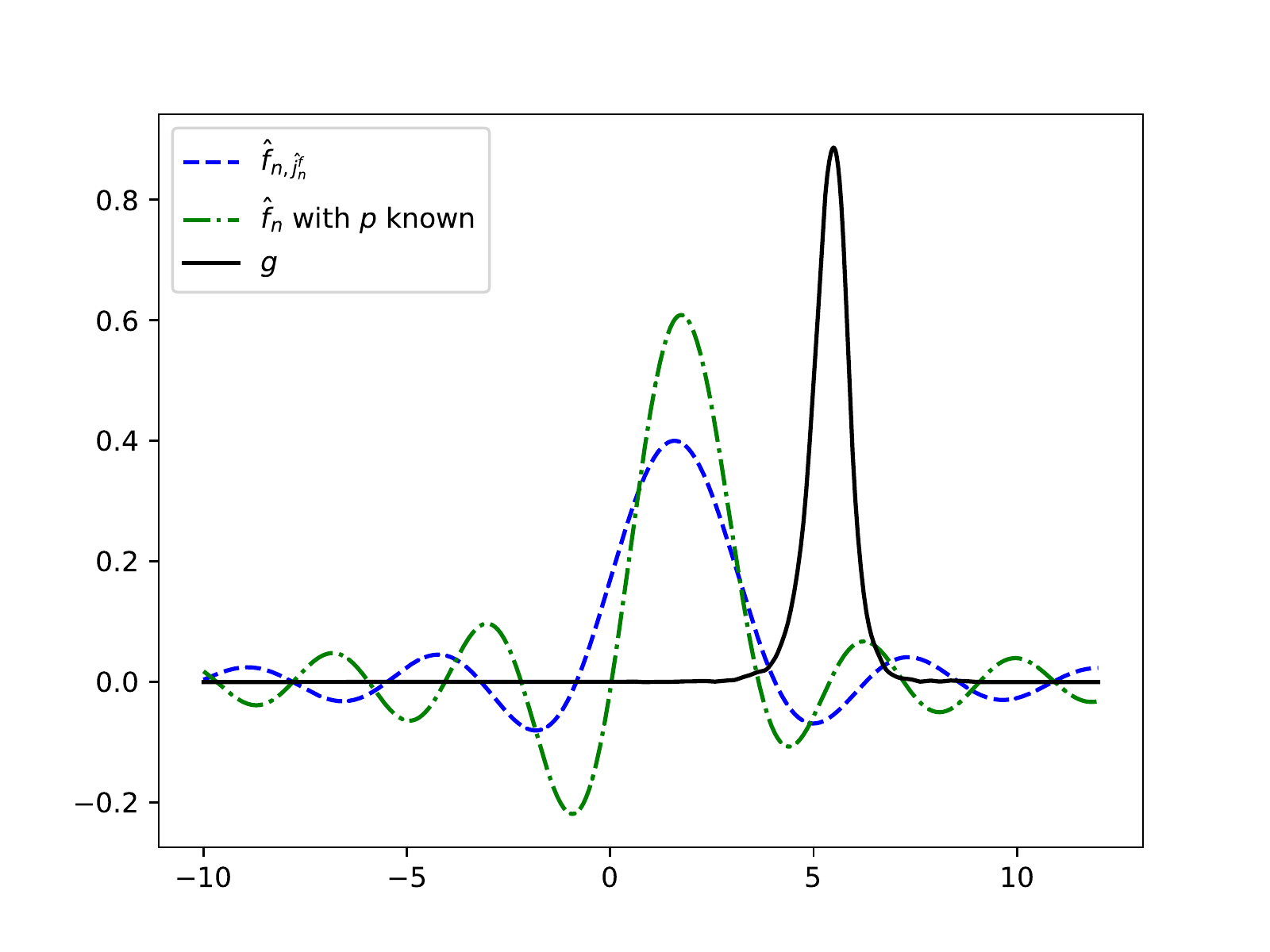}\\
(d)
\end{minipage}
\hfill
\begin{minipage}{0.3\linewidth}\centering
\includegraphics[scale=0.32]{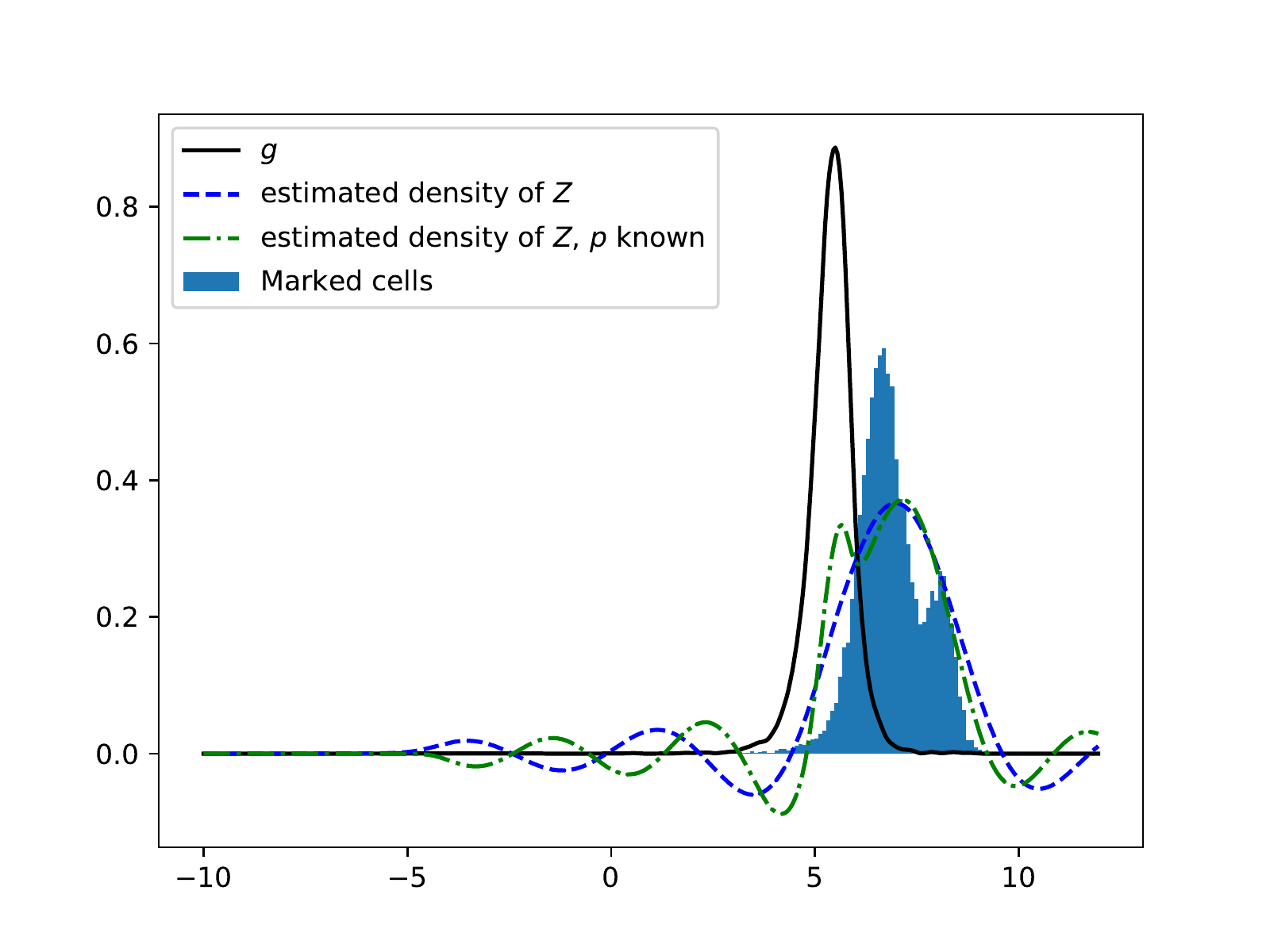}\\
(e)
\end{minipage}
\hfill
\begin{minipage}{0.3\linewidth}\centering
\includegraphics[scale=0.32]{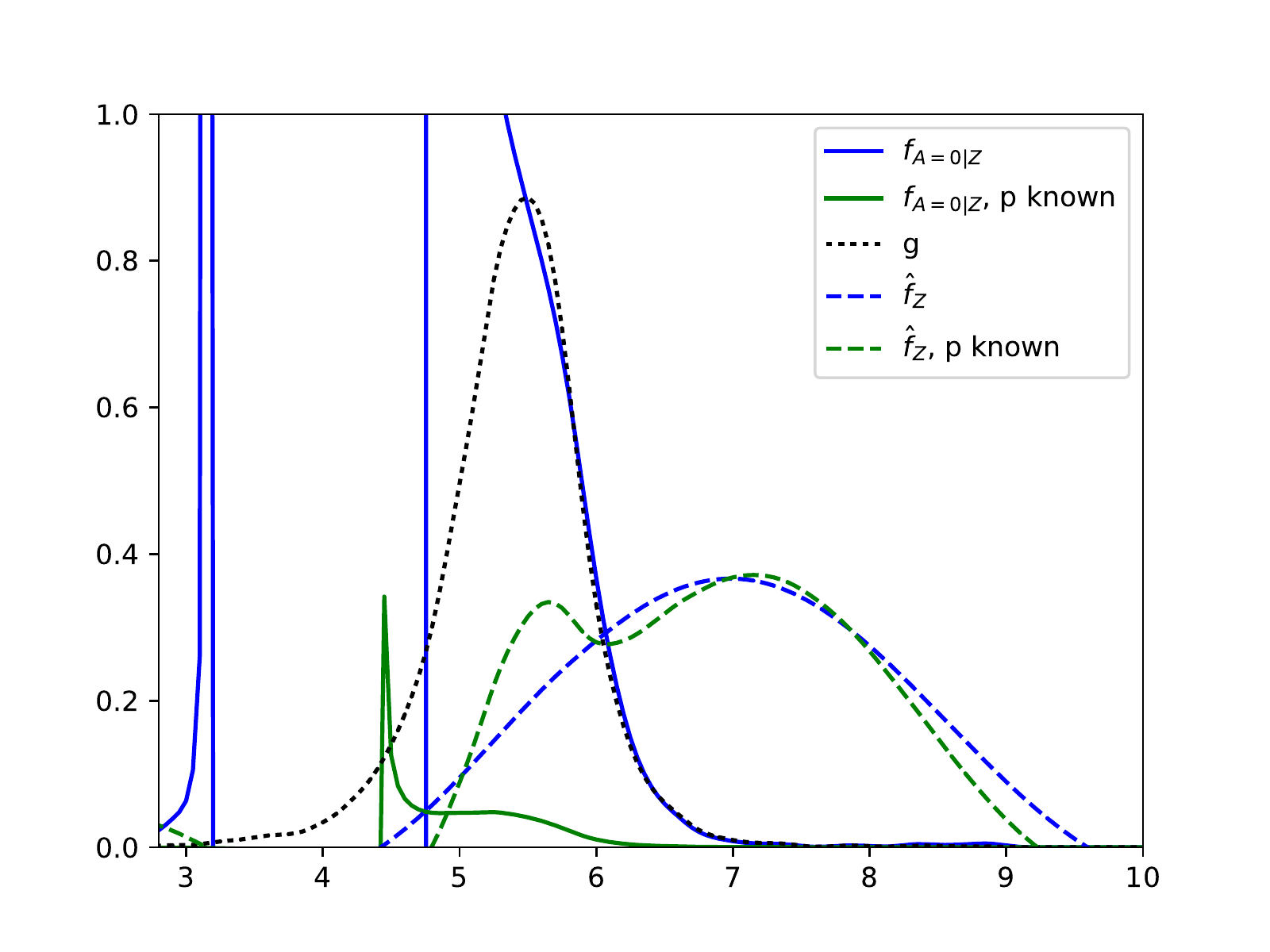}\\
(f)
\end{minipage}
\hfill
\caption{\small{Results of the Lepskii procedure on real data. (a) Histogram of the fluorescence of treated cells and estimated probability distribution for the fluorescence of un-treated cells. (b) Estimation of $p$ on an increasing sub-sample of the data (the data being permutated 5 times). (c) In red, empirical Fourier transform $\hat{\phi}_Z$ and two estimators computed respectively with the values $\hat{f}_{n,\hat{j}^f_n}$ and $\hat{p}^{(1)}_{n,\hat{j}^f_n}$ obtained with the Lepskii procedure for $f$ (in blue) and with the values obtained when $p$ is first estimated and then used in the estimation of $f$ (in green). (d) Comparison of the obtained estimators of $f$. (e) Estimated density of $Z$ computed using the two different strategies for estimating $f$ and comparison with the histogram of data. (f) Estimators of the conditional density
$f_{A=0|Z}(z)$. }}
\label{fig:test3}
\end{figure}

\paragraph{Estimation of $f$}
Next, we focus on the estimation of the density $f$, which represents here the distribution of fluorescence intensity due to the presence of CD27 on the cells' surface. We use two different estimation strategies: 
\begin{itemize} 
\item First, we apply the global Lepskii estimation procedure introduced in Section \ref{subsec:lepskif}.
\item Second, we use a similar approach but we consider that the value of $p$ is known and given by the previously estimated $\hat{p}_{n,\hat{j}_n,p}$ above. We then use the Lepskii procedure with a fixed value of $p$ as described in Section \ref{subsec:lepskif}.
\end{itemize}
In both cases we then compute an approximation of the Fourier transform $\Phi_Z$ of the mixed signal and compared it with the empirical estimator $\hat{\Phi}_{Z,n}$ in Figure \ref{fig:test3}(c). In particular, in Figure \ref{fig:test3}(c), we show several colored curves:
\begin{itemize}
\item  the red curve is the empirical estimator $\hat{\Phi}_{Z,n}$
\item  the blue dotted curve is the estimator derived from the global Lepskii procedure
\item the green dotted curve corresponds  to the modified Lepskii procedure for $f$ when $p$ is previously estimated).
 \end{itemize}
 
  We observe that the estimator of $f$ computed with $p=0.31$ seems to behave better, which is not surprising since the Lepskii procedure for $f$ does not provide any theoretical guaranty for estimating $p$ (we obtained $\hat{p}^{(1)}_{n,\hat{j}_n^f}=0.081$, which seems very low and not realistic for our biological framework).
Then we calculate the inverse transform of both estimators (see Figure \ref{fig:test3}(d) ), compute the estimated density of $Z$ and compare it with the observed histogram (Figure \ref{fig:test3}(e)). We observe first that our estimation can take negative values (due to the bandwidth parameter in the estimation) and secondly that it fails to capture the exact shape of the density, however it reproduces reasonably well the mean and the variance.

Finally, we study the  probability that a cell expresses CD27 on its surface conditionally to the value of the observed fluorescence. We compute estimators of this conditional density with the Bayes formula:
$$f_{A=0|Z}(z) = \frac{pg(z)}{pg(z)+(1-p)f\ast g(z)}$$
The results are drawn on Figure \ref{fig:test3}(f) for both estimators of $f$. Since our estimation of $f$ can take negative values and vanish, we observe artificial singularities that should be omitted for the biological interpretation. We have not pushed further our investigations from a biological point of view since the goal of our paper is mainly theoretical.

\appendix
\section{Technical results}\label{sec:appendix}
\subsection{Bernstein inequality}
We first state a classical concentration inequality that may be found  for example in \cite{BLM}.
\begin{theo}[Bernstein inequality]\label{theo:bernstein}
Let $X_1,\ldots,X_n$ be independent random variables with finite variance such that $|X_i| \leq b$ for some $b > 0$ almost surely for all $i \leq n$. Let $S=\sum_{i=1}^n X_i - \mathbb{E}[X_i]$ and $v=\sum_{i=1}^n \mathbb{E}[X_i^2]$. Then
$$
\mathbb{P}\left( |S| \geq t \right) \leq  2 \exp\left( -\frac{t^2}{2(v+bt/3)}\right).
$$
\end{theo}
\par
\subsection{Non adaptive estimation}
\label{app1}
In this section, we sketch  the proof of Theorem \cite{vanes2} with $\tau_n = \log(n)^{-a}$ and $a>0$ instead of the choice $\tau_n=\log(3n)^{-1}$ given in \cite{vanes2}, in order to control $\E(\|\hfjn-f\|_{_2})$.
\begin{pro}
\label{prop:mieux}
If $\Hs$ and $\Hnu$ hold with $\nu > 1$, then the choices $h_{n,j} = n^{-\frac{1}{2s_j+2\nu}}$ in \eqref{eq:defpn}, $\tau_n=\log(n)^{-a}$ in \eqref{eq:taun} and $\delta_{n,j} = n^{-1/(2s_j+2\nu+1)}$ in \eqref{eq:def_deltanl} lead to an estimator $\hfnj$ that satisfies the consistency rate:
$$
\mathbb{E} [\|\hfnj-f\|^2_2] \le C(s,R) n^{-\frac{2\min(s,s_j)}{2s_j+2\nu+1}},
$$
where $C(s,R)$ is a positive constant depending continuously in $s$ and $R$.
\end{pro}
\begin{proof}
The bias-variance decomposition and Fubini's Theorem yield
$$\mathbb{E} [\|\hfnj-f\|^2_2]\le \int_{-\infty}^{\infty} \left(\E(\hfnj(t)-f(t))\right)^2 dt+\int_{-\infty}^{\infty} \VV(\hfnj(t)) dt
$$
\paragraph{Integrated bias}
In this section, we introduce more explicit notations, we set 
$\hat{f}_{n,h_{n,j},\delta_{n,j}}=\hfnj$ defined by \eqref{eq:hatf} and set $\hat{f}_{n,\delta_{n,j}}$ the estimator of $f$ when $p$ is known :
$$\hat{f}_{n,\delta_{n,j}}=\frac{1}{2\pi}\int e^{-itx}\frac{\hpz{2}(t)-p \PU(t)}{(1-p)\PU(t)}\PQ(\delta_{n,j}t) dt.
$$
We also define
$$\tilde{f}_{n,\delta_{n,j}}=\frac{1}{2\pi}\int e^{-itx}\frac{\hpz{2}(t) }{\PU(t)}\PQ(\delta_{n,j}t) dt.
$$
Then
\bea
\int\left(\E(\hfnj-f)\right)^2 &= \int\left(\E(\hat{f}_{n,h_{n,j},\delta_{n,j}}-f)\right)^2\\
& \le 2 \int\left(\E(\hat{f}_{n,h_{n,j},\delta_{n,j}}-\hat{f}_{n,\delta_{n,j}})\right)^2 +2\int\left(\E(\hat{f}_{n,\delta_{n,j}}-f)\right)^2
\eea
Using the Parseval inequality, then assumptions \eqref{eq:kernel_f} and $f\in\Hs$ we deduce that the second term behaves like:
\bea
\int\left(\E(\hat{f}_{n,\delta_{n,j}}-f)\right)^2
&\le\frac{1}{2\pi}\int_{-\infty}^{\infty} |\Phi_f(t)|^2 \frac{|\PQ(\delta_{n,j}t)-1|^2}{(\delta_{n,j}t)^{2s}}(\delta_{n,j}t)^{2s} dt,\\
&\le MR\delta_{n,j}^{2s}.
\eea
Then, for the first term we obtain that:
\bea
\int\left(\E(\hat{f}_{n,h_{n,j},\delta_{n,j}}-\hat{f}_{n,\delta_{n,j}})\right)^2
&\le 2\int\E\left[ \tilde{f}_{n,\delta_{n,j}}\frac{(\hpnj^{(1)}-p)}{(1-p)(1-\hpnj^{(1)})} \right]^2 \\
&\hspace{2cm}+2\E\left[\frac{(\hpnj^{(1)}-p)}{(1-p)(1-\hpnj^{(1)})} \right]^2\int \left(\frac{1}{\delta_{n,j}}Q(t\delta_{n,j}^{-1} ) \right)^2 dt\\
&\le 2\E\left[\frac{(\hpnj^{(1)}-p)}{(1-p)(1-\hpnj^{(1)})} \right]^2 \int\E\left[ \tilde{f}_{n,\delta_{n,j}}\right]^2\\
&\hspace{2cm}+2\E\left[\frac{(\hpnj^{(1)}-p)}{(1-p)(1-\hpnj^{(1)})} \right]^2\int \left(\frac{1}{\delta_{n,j}}Q(t\delta_{n,j}^{-1} ) \right)^2 dt
\eea
Assumptions \eqref{eq:kernel_f} on the kernel $Q$ and the Parseval inequality lead to:
$$\int_{-\infty}^{\infty} \left(\frac{1}{\delta_{n,j}}Q(t\delta_{n,j}^{-1} ) \right)^2 dt =\delta_{n,j}^{-1}\int_{-1}^{1} Q(t)^2dt<\infty.
$$
Moreover, we use a concentration inequality similar to Proposition \ref{prop:concentration_p} stated in Lemma \ref{lem:A2}.
It remains to compute 
\bea
\int \E[\tilde{f}_{n\delta_{n,j}}]^2=\int \VV(\tilde{f}_{n\delta_{n,j}})+\int \E[\tilde{f}_{n\delta_{n,j}}^2].
\eea
From the definition of $\tilde{f}_{n\delta_{n,j}}$, we remark that 
\bea
\tilde{f}_{n\delta_{n,j}}(x)=\frac{1}{\delta_{n,j}(n-n')}\sum_{k=n'}^n W_{n,j}\left(\frac{x-Z_k}{\delta_{n,j}}\right)
\eea where $W_{n,j}(x)=\int_{-1}^1 e^{-itx}\PQ(t)\PU(t\delta_{n,j}^{-1})^{-1}dt$.
Then 
$$\int \VV(\tilde{f}_{n\delta_{n,j}})\le \frac{1}{\delta_{n,j}^2(n-n')}\int_{-\infty}^{\infty}\E\left( W_{n,j}\left( \frac{x-Z}{\delta_{n,j}} \right)^2 \right)dx$$
A straightforward computation using that $g\in\Hnu$ leads to:
\bea
\int_{-\infty}^{\infty}\E\left( W_{n,j}\left( \frac{x-Z}{\delta_{n,j}} \right)^2 \right)dx
&\le \frac{4\delta_{n,j}^{1-2\nu}}{d_2}\int_{-1}^{1}\lvert\PQ\lvert^2(\delta_{n,j}^{\nu}+t^{\nu})^2 dt
\eea
Thus for some constant $C$:
\be
\int \VV(\tilde{f}_{n\delta_{n,j}})\le C\delta_{n,j}^{-1-2\nu}(n-n')^{-1},
\ee
which entails as $s_j>0$ that $\VV(\tilde{f}_{n\delta_{n,j}})\to0$.\\
It remains to bound
$\int \E[\tilde{f}_{n,\delta_{n,j}}^2]
$. The Parseval inequality and the fact that $|\PX|\le1$ yields:
\be
\int \E[\tilde[f_{n,\delta_{n,j}}^2]\le \delta_{n,j}^{-1}\int_{-1}^1|\PQ(t)|^2dt.
\ee
Therefore, combining the previous inequality we deduce that:
\bea
\int\left(\E(\hfnj-f)\right)^2 
&\le 2MR\delta_{n,j}^{2s} +4C(s,R)(n-n')^{-\frac{2\min(s,s_j)+1}{2s_j+2\nu}}\\
&\hspace{1cm}\times \left( \delta_{n,j}^{-1}\int Q(t)^2dt+ C\delta_{n,j}^{-1-2\nu}(n-n')^{-1}+ \delta_{n,j}^{-1}\int_{-1}^1|\PQ(t)|^2dt\right)
\eea
The second term of the r.h.s. is of order $h_{n,j}^{2\min(s,s_j)+1}\delta_{n,j}^{-1}$ which is negligible before $\delta_{n,j}^{2\min(s,s_j)}$. This entails that:
\ben
\int\left(\E(\hfnj-f)\right)^2 
\le C(s,R)(n)n^{-\frac{2\min(s,s_j)}{2s_j+2 \nu+1}}.
\een
\paragraph{Integrated variance}
Let us first remark that:
\bea
\int_{-\infty}^{\infty}\VV(\hfnj)
&=\int_{-\infty}^{\infty}\VV(\hat{f}_{n,h_{n,j},\delta_{n,j}})\\
&\le 2\int_{-\infty}^{\infty}\VV(\hat{f}_{n,h_{n,j}}) +2\int_{-\infty}^{\infty}\VV(\hat{f}_{n,h_{n,j},\delta_{n,j}}-\hat{f}_{n,h_{n,j}}).
\eea
We have proved above that 
\be
2\int_{-\infty}^{\infty}\VV(\hat{f}_{n,h_{n,j}}) \le C\delta_{n,j}^{-1-2\nu}(n-n')^{-1}.
\ee
It remains to consider the second term.
We use once again an auxiliary sequence $\Psi_n \to0$:
\bea
\int_{-\infty}^{\infty}\VV(\hat{f}_{n,h_{n,j},\delta_{n,j}}-\hat{f}_{n,\delta_{n,j}}) 
&\le \int_{-\infty}^{\infty}\E[(\hat{f}_{n,h_{n,j},\delta_{n,j}}-\hat{f}_{n,\delta_{n,j}})^2]\\
&\le \int_{-\infty}^{\infty}\E[(\hat{f}_{n,h_{n,j},\delta_{n,j}}-\hat{f}_{n,\delta_{n,j}})^2\un_{|\hpnj-p|>\Psi_n}]\\
&\quad+\int_{-\infty}^{\infty}\E[(\hat{f}_{n,h_{n,j},\delta_{n,j}}-\hat{f}_{n,\delta_{n,j}})^2\un_{|\hpnj-p|\le \Psi_n}].\\
\eea
Similarly as for the integrated bias we obtain that:
\bea
\int_{-\infty}^{\infty}\E[(\hat{f}_{n,h_{n,j},\delta_{n,j}}-\hat{f}_{n,\delta_{n,j}})^2\un_{|\hpnj-p|>\Psi_n}]
&\le 2\int\E\left[\frac{(\hpnj^{(1)}-p)^2}{(1-p)^2(1-\hpnj^{(1)})^2} \tilde{f}_{n,\delta_{n,j}}^2\un_{|\hpnj^{(1)}-p|>\Psi_n}\right]\\
&\quad +2\E\left[\frac{(\hpnj^{(1)}-p)^2}{(1-p)^2(1-\hpnj^{(1)})^2}\un_{|\hpnj^{(1)}-p|>\Psi_n} \right]\int \left(\frac{1}{\delta_{n,j}}Q(t\delta_{n,j}^{-1} )\right)^2  dt.\\
\eea
The second term of the r.h.s. is bounded by
$$2\delta_{n,j}^{-1}\frac{2(1-\tau_n)^2+2p^2}{\tau_n^2(1-p)}\left(\int_{-1}^1Q(t)dt\right)\Pro(|\hpnj^{(1)}-p|>\Psi_n).$$
We prove below in \eqref{app:concentr_hp1} that $\Pro(|\hpnj^{(1)}-p|>\Psi_n)$ decreases exponentially fast to $0$ when 
$\psi_n$ is chosen as $\psi_n =n^{-1/(2(2s_{\max}+2\nu))}.$
Thus, the previous term is $o\left(n^{-\frac{2\min(s,s_j)}{2s_j+2\nu+1}}\right)$.

\noindent
Concerning the other term, we apply the Fubini theorem and obtain that:
\bea
\int\E\left[\frac{(\hpnj^{(1)}-p)^2}{(1-p)^2(1-\hpnj^{(1)})^2} \tilde{f}_{n,\delta_{n,j}}^2\un_{|\hpnj^{(1)}-p|>\Psi_n}\right] 
&= \E\left[\frac{(\hpnj^{(1)}-p)^2}{(1-p)^2(1-\hpnj^{(1)})^2} \un_{|\hpnj^{(1)}-p|>\Psi_n}\int\tilde{f}_{n,\delta_{n,j}}^2\right]\\
&=\E\left[\frac{(\hpnj^{(1)}-p)^2}{(1-p)^2(1-\hpnj^{(1)})^2} \un_{|\hpnj^{(1)}-p|>\Psi_n}\int\frac{|\hPZ(t)\PQ(\delta_{n,j}t)|^2}{|\PU(t)|^2}dt\right]\\
&\le \frac{2+2p^2}{\tau_n^2(1-p)^2}\frac{4\delta_{n,j}^{-1-2\nu}}{d_2^2}\int_{-1}^1[\PQ(t)|^2(1+t^\nu)^2dt \Pro(|\hpnj^{(1)}-p|>\Psi_n).
\eea
Hence, we deduce that similarly
$$ \int_{-\infty}^{\infty}\E[(\hat{f}_{n,h_{n,j},\delta_{n,j}}-\hat{f}_{n,\delta_{n,j}})^2\un_{|\hpnj^{(1)}-p|>\Psi_n} \le o\left(n^{-\frac{2\min(s,s_j)}{2s_j+2\nu+1}}\right).$$
We finally consider the event where $|\hpnj^{(1)}-p|\le\Psi_n$ and use the independence between $\hpnj^{(1)}$ and $\tilde{f}_{n,\delta_{n,j}}$ to obtain
\bea
\int_{-\infty}^{\infty}\E[(\hat{f}_{n,h_{n,j},\delta_{n,j}}-\hat{f}_{n,\delta_{n,j}})^2\un_{|\hpnj^{(1)}-p|\le\Psi_n}]
&\le 2\int\E\left[\frac{(\hpnj^{(1)}-p)^2}{(1-p)^2(1-\hpnj^{(1)})^2}\un_{|\hpnj^{(1)}-p|\le\Psi_n}\right]\E\left[ \tilde{f}_{n,\delta_{n,j}}^2\right]\\
&\quad +2\E\left[\frac{(\hpnj^{(1)}-p)^2}{(1-p)^2(1-\hpnj^{(1)})^2}\un_{|\hpnj^{(1)}-p|\le\Psi_n} \right]\int \left(\frac{1}{\delta_{n,j}}Q(t\delta_{n,j}^{-1} )\right)^2  dt.\\
\eea
\end{proof}
From Lemma \ref{lem:A2} and an argument similar to the one used for the integrated bias, we conclude that this last term is equal to $O\left(n^{-\frac{2\min(s,s_j)}{2s_j+2\nu+1}}\right)$, which concludes the proof of Proposition \ref{prop:mieux}. \hfill$\square$

\begin{lem} Assume $s_j\ge1/2$ and that $f\in\Hs$, $g\in \Hnu$, with the choices of $h_{n,j}=n^{-1/(2s_j+2\nu)}$ and $\tau_n= \log(n)^{-a}$ then
\label{lem:A2}
\be\E\left[\frac{(\hpnj^{(1)}-p)}{(1-p)(1-\hpnj^{(1)})} \right]^2
\le C(s,R)n^{-\frac{2\min(s,s_j)+1}{2s_j+2\nu}}.
\ee
\end{lem}

\begin{proof}
Let us first remark that using the Cauchy-Schwarz inequality and a truncation strategy, we have that:
\bea
\E\left[\frac{(\hpnj^{(1)}-p)}{(1-p)(1-\hpnj^{(1)})} \right]^2 \le \frac{1}{\tau_n^2(1-p)^2}\E[(\hpnj^{(1)}-p)^2].
\eea
Moreover using \eqref{eq:biais} and \eqref{eq:var} 
\bea
\E[(\hpnj^{(1)}-p)^2] \le 2C_sR^2h_{n,j}^{2s+1} +\frac{cte(s)}{n h_{n,j}^{2\nu-1}},
\eea
where $cte(s)$ is a constant which depend continuously in $s$. Therefore, with our choice of $h_{n,j}$ in \eqref{def:hnj} we deduce that there exist a constant $\phi(s,R)$ depending continuously on $s$ an d $R$ such that
\be
\E[(\hpnj^{(1)}-p)^2] \le \phi(s,R)n^{-\frac{2\min(s,s_j)+1}{2s_j+2_nu}},
\ee
and thus
\bea
\E\left[\frac{(\hpnj^{(1)}-p)}{(1-p)(1-\hpnj^{(1)})} \right]^2 \le \frac{1}{\tau_n^2(1-p)^2}\phi(s,R)n^{-\frac{2\min(s,s_j)+1}{2s_j+2_nu}}.
\eea
With this simple reasoning our upper bound depend on the truncation $\tau_n$. Therefore, we will refine the result by using an auxiliary sequence $\Psi_n\to0$ and split the events into two sub-cases:
\bea
\E\left[\frac{(\hpnj^{(1)}-p)}{(1-p)(1-\hpnj^{(1)})} \right]^2
&\le \E\left[\frac{(\hpnj^{(1)}-p)^2}{(1-p)^2(1-\hpnj^{(1)})^2} \right]\\
&\le \E\left[\frac{(\hpnj^{(1)}-p)^2}{(1-p)^2(1-\hpnj^{(1)})^2} \un_{\lvert \hpnj^{(1)}-p\lvert \le \Psi_n}\right]+ \E\left[\frac{(\hpnj^{(1)}-p)^2}{(1-p)^2(1-\hpnj^{(1)})^2} \un_{\lvert \hpnj^{(1)}-p\lvert >\Psi_n}\right]\\
&\le \frac{1}{(1-p)^2(1-p-\Psi_n)^2}\phi(s,R)n^{-\frac{2\min(s,s_j)+1}{2s_j+2_nu}} +\frac{1}{(1-p)^2\tau_n^2}\Pro(\lvert \hpnj^{(1)}-p\lvert >\Psi_n).
\eea
In the following, our aim is to calibrate the sequence $\Psi_n$ such that the second term is negligible comparing to the first one. We have studied concentration inequalities in Section \ref{subsec:concentration_p} for a non truncated version of $\hpjn^{(1)}$, that we denote here $\tilde{p}_{n,j}^{(1)}$.
Then
\bea
\Pro(\lvert \hpnj^{(1)}-p\lvert >\Psi_n)
&\le \Pro(\lvert \hpnj^{(1)}-\tilde{p}_{n,j}^{(1)}\lvert >\Psi_n)+\Pro(\lvert \tilde{p}_{n,j}^{(1)}-p\lvert >\Psi_n)\\
&\le \E(\un_{\tau_n>\Psi_n/2}\un_{\tilde{p}_{n,j}^{(1)}>1-\tau_n})+\Pro(\lvert \tilde{p}_{n,j}^{(1)}-p\lvert >\Psi_n/2),\\
& \le \E\Bigl(\un_{\tau_n>\Psi_n/2}\un_{\tilde{p}_{n,j}^{(1)}>1-\tau_n}\Bigr)+
\Pro\left( \left|\frac{1}{n}\sum_{k=1}^n\xi_{k,j} \right| >\Psi_n/4\right)\quad+\Pro( C_{s} R h_{n,l}^{s+1/2}>\Psi_n/4).
\eea
where the random variables $\xi_{k,j}$ are defined in \eqref{eq:step1}
Therefore 
\bea
\Pro(\lvert \hpnj^{(1)}-p\lvert >\Psi_n)
&\le \Pro(\lvert \hpnj^{(1)}-\tilde{p}_{n,j}\lvert >\Psi_n)+\Pro(\lvert \tilde{p}_{n,j}^{(1)}-p\lvert >\Psi_n)\\
&\le \E(\un_{\tau_n>\Psi_n/2}\un_{\tilde{p}_{n,j}^{(1)}>1-\tau_n})+\Pro(\lvert \tilde{p}_{n,j}^{(1)}-p\lvert >\Psi_n/2),\\
& \le \E\Bigl(\un_{\tau_n>\Psi_n/2}\un_{\tilde{p}_{n,j}^{(1)}>1-\tau_n}\Bigr)+
\Pro\left( \left|\frac{1}{n-n'}\sum_{k=n'}^n\xi_{k,j} \right| >\Psi_n/4\right)\\
&\quad+\Pro( C_{s} R h_{n,l}^{s+1/2}>\Psi_n/4).
\eea Using the Bernstein inequality, we then have:
\bea
\Pro(\lvert \hpnj^{(1)}-p\lvert >\Psi_n)
&\le \un_{\tau_n>\Psi_n/2}\Pro\Bigl(\tilde{p}_{n,j}^{(1)}>1-\tau_n\Bigr)+
 \exp \left( - \frac{(n-n') \Psi_n^2}{8 \left( C' h_{n,l}^{-2\nu+1}+ C h_{n,l}^{-\nu} \Psi_n/3)\right)}\right)\\
 &\quad+\un_{ C_{s} R h_{n,l}^{s+1/2}>\Psi_n/4}.
\eea
We choose $\Psi_n=n^{-1/2(2s_{max}+2\nu)}$ such that the last term of the r.h.s. is always null. Thus, it remains to consider $\Pro(\tilde{p}_{n,j}>1-\tau_n)$. Applying similarly Bernstein inequality leads to:
\bea
\Pro(\tilde{p}_{n,j}^{(1)}>1-\tau_n)\le \Pro(\lvert\tilde{p}_{n,j}^{(1)}-p\lvert>1-\tau_n-p)\le  \exp \left( - \frac{(n-n') (1-\tau_n-p)^2}{8 \left( C' h_{n,l}^{-2\nu+1}+ C h_{n,l}^{-\nu} (1-\tau_n-p)/3)\right)}\right)
\eea
The two previous inequalities yield:
\bean
\label{app:concentr_hp1}
\Pro(\lvert \hpnj^{(1)}-p\lvert >\Psi_n)
&\le \exp \left( - \frac{(n-n') (1-\tau_n-p)^2}{8 \left( C' h_{n,l}^{-2\nu+1}+ C h_{n,l}^{-\nu} (1-\tau_n-p)/3)\right)}\right).\\
&\quad+
 \exp \left( - \frac{(n-n') \Psi_n^2}{8 \left( C' h_{n,l}^{-2\nu+1}+ C h_{n,l}^{-\nu} \Psi_n/3)\right)}\right)
\eean
It is clear that $\tau_n^{-2}\Pro(\lvert \hpnj^{(1)}-p\lvert >\Psi_n) = o \left( n^{-1/2(2s_{max}+2\nu)} \right)$, which concludes the proof. \hfill $\square$
\end{proof}
\par

 \bibliographystyle{alpha}
\bibliography{CGR}
\vskip2cm
\hskip70mm\box5

\end{document}